\documentclass[11pt]{amsart}
\usepackage{amsmath,amsfonts,amssymb,amsthm,epsfig,verbatim,ifthen,graphicx}
\usepackage[usenames,dvips]{color}
\usepackage[all]{xy}
 \numberwithin{equation}{section}

\theoremstyle{definition} 
\newtheorem{definition}{Definition}[section]
\newtheorem{theorem}[definition]{Theorem}
\newtheorem{proposition}[definition]{Proposition}
\newtheorem{lemma}[definition]{Lemma}

\theoremstyle{definition}
\newtheorem{corollary}[definition]{Corollary}

\newtheorem{example}[definition]{Example}
\newtheorem{remark}[definition]{Remark}

\newcommand{\N}{\mathbb{N}}
\newcommand{\Z}{\mathbb{Z}}
\newcommand{\R}{\mathbb{R}}
\newcommand{\p}{\mathcal{P}}
\newcommand{\T}{\mathbb{T}}

\begin{document}

\title
{Existence of a measurable saturated compensation function between
subshifts and its applications}
\author{Yuki Yayama}
\address{Centro de Modelamiento Matem\'{a}tico, Universidad de Chile, Av. Blanco Encalada 2120 Piso 7 Santiago, Chile} \email{yyayama@dim.uchile.cl}
\begin{abstract}
We show the existence of a bounded Borel measurable saturated
compensation function for a factor map between subshifts. As an
application, we find the Hausdorff dimension and measures of full
Hausdorff dimension for a compact invariant set of an expanding
nonconformal map on the torus given by an integer-valued diagonal
matrix. These problems were studied in \cite{Yayama} for a compact
invariant set whose symbolic representation is a shift of finite
type under the condition of the existence of a saturated
compensation function. We extend the results by presenting a
formula for the Hausdorff dimension for a compact invariant set
whose symbolic representation is a subshift without the condition
and characterizing the invariant ergodic measures of full
dimension as the ergodic equilibrium states of a constant multiple
of a measurable compensation function. For a compact invariant set
whose symbolic representation is a topologically mixing shift of
finite type, we study uniqueness and the properties for the unique
invariant ergodic measure of full dimension by using a measurable
compensation function. Our positive results narrow the possibility
of where an example having more than one measure of full dimension
can be found.
\end{abstract}
\maketitle
\pagestyle{myheadings}
\markright{Existence of a measurable saturated compensation function between
subshifts}
\section{Introduction}
Let $(X,\sigma_X)$ and $(Y,\sigma_Y )$ be subshifts. Let
$\pi:(X,\sigma_X) \rightarrow (Y,\sigma_Y )$ be a factor map,
i.e., $\pi$ is a continuous and surjective function that satisfies
$\pi\circ \sigma_X= \sigma_Y\circ\pi$. A function $F\in C(X)$ is a
\em{compensation function }\em for $(\sigma_X, \sigma_Y, \pi)$ if
$$\sup_{\mu\in M(X, \sigma_X)}\{h_{\mu}(\sigma_X)+\int (F+\phi\circ \pi) d \mu\}=\sup_{m\in M(Y, \sigma_Y)}\{h_{m}(\sigma_Y)+\int \phi dm\}
$$ for all $\phi\in C(Y)$. If
$F=G\circ\pi$ with $G\in C(Y)$, then $G\circ \pi$ is a
\em{saturated compensation function}\em. The concept of
compensation function was introduced by Boyle and Tuncel \cite{BT}
and studied by Walters \cite{Wcom} in connection with relative
pressure. If $(X, \sigma_X), (Y, \sigma_Y)$ are shifts of finite
type, then there always exists a compensation function
\cite{Wcom}. An example of a factor map between shifts of finite
type without a saturated compensation function was given by Shin
\cite{S3} and necessary and sufficient conditions for the
existence of a saturated compensation function were studied by
Shin \cite{S1}. In this paper, we show that for
$(X,\sigma_X),(Y,\sigma_Y)$ subshifts there always exists a
bounded Borel measurable saturated compensation function and
characterize the equilibrium states (Theorem \ref{thm1}). Using
measurable saturated compensation functions, we then study for a
fixed $\alpha>0$ the measures that maximize the weighted entropy
functional $\phi_{\alpha}=h_{\mu}(\sigma_{X}) +\alpha h_{\pi
\mu}(\sigma_{Y}).$ Finding such measures is useful in problems on
Hausdorff dimension (see \cite{GP}). When there is a saturated
compensation function between subshifts, the measures that
maximize $\phi_{\alpha}$ are, according to  Shin \cite{S2}, the
equilibrium states of a constant multiple of a saturated
compensation function. We extend the results for any subshifts
without the existence of a saturated compensation function
(Proposition \ref{sftwentropy} and Theorem \ref{wentropy}).

As an application, we study the problem on dimensions of compact
invariant sets of nonconformal expanding maps, in particular, we
consider the endomorphism of the torus given by $T(x,y)=(lx
\textnormal{ mod }1, my \textnormal{ mod }1 ), l>m\geq 2, l,m \in
\N$. Throughout this paper, by a measure $\mu$ of full dimension
for a compact $T$-invariant set $K$, we mean that $\mu$ is a Borel
probability measure of full Hausdorff dimension for $K$ and
$\mu(K)=1$.

In \cite{Yayama}, saturated compensation functions are used to
find the Hausdorff dimension and measures of full dimension of
some compact $T$-invariant sets, giving a systematic way to
approach the problem (see \cite{Yayama} for more details). As a
result, studying the existence of a saturated compensation
function was one of the main difficulties. We circumvent this
issue by using a {\em measurable} saturated compensation function
(which always exists between subshifts) instead of a {\em
continuous} saturated compensation function.

The compact invariant sets considered in this paper are a
generalization of the sets studied by Bedford \cite{B} and
McMullen \cite{Mc}. They independently answered the question on
Hausdorff dimension for compact $T$-invariant subsets which we
call {\em NC carpets} (see section \ref{application} for the
definition). Uniqueness of the invariant ergodic measure of full
dimension for an NC carpet was shown by Kenyon and Peres
\cite{KP}.  Using a coding map constructed by a Markov partition
for $T,$ one obtains a symbolic representation of the carpet which
is a full shift on finitely many symbols. A more general set whose
symbolic representation is a shift of finite type was considered
in \cite{Yayama}. Such a set is called an \em SFT-NC carpet. \em
The results in \cite{Yayama} gave a formula for the Hausdorff
dimension for an SFT-NC carpet for which a saturated compensation
function exists (when it is represented in symbolic dynamics) and
the $T$-invariant ergodic measures of full dimension are
identified as the ergodic equilibrium states of a constant
multiple of a saturated compensation function.

In this paper, we extend the formula for a compact $T$-invariant
set whose symbolic representation is a subshift
$(X,\sigma_X)$(Theorem \ref{application1}). We can do this because
there always exists a measurable saturated compensation function
between subshifts. Then the $T$-invariant ergodic measures of full
dimension are characterized as the ergodic equilibrium states of a
constant multiple of a measurable saturated compensation function,
which turns out to be a potential $-\Phi\circ\pi$ on
$(X,\sigma_X),$
 where $\Phi$ is a sequence of real-valued bounded Borel measurable functions on $(Y,\sigma_Y)$
 that satisfies the subadditivity condition (Theorem
\ref{application1}).
  For a compact $T$-invariant set whose symbolic representation is a topologically mixing shift of finite type,
  we study uniqueness of the $T$-invariant ergodic measure of full dimension and the properties of
  the unique
  measure (Theorem
\ref{uniqueness}). In this case, a measurable saturated
compensation
  function can be replaced by a much nicer function (see Corollary \ref{nice}). Using
  it, we study the equilibrium states for a potential
$-\Phi\circ\pi$ on $(X,\sigma_X),$ where $\Phi$ is a sequence of
real-valued continuous functions on $(Y,\sigma_Y)$ that satisfies
the subadditivity condition (Theorem \ref{HDDFORSFT}). $\Phi$
becomes an almost additive potential on $(Y,\sigma_Y)$ under some
conditions (see page \pageref{aadditiveinourcase}). In order to
study the equilibrium states for a potential $-\Phi\circ\pi$, we
show that $\Phi$ has a unique equilibrium state which is Gibbs
regardless of the condition of the almost additivity. These
results generalize the results in \cite{Yayama}. We use the work
on almost additive potentials by Barreira \cite{B2006} and Mummert
\cite{m2006} and the work on subadditive potentials by Cao, Feng
and Huang \cite{CFH}. Our positive results narrow the possibility
of where an example having more than one measure of full dimension
can be found.

\section{Background}
We give a brief overview of the recent results in pressure theory
for almost additive potentials and subadditive potentials
\cite{{B2006},{m2006},{CFH}}. These generalize the work of Ruelle
and Walters on the variational principle for continuous functions.
For notation and terminology not explained here including the
definitions of a full shift and shift of finite type, see
\cite{LM, Wtext}. Throughout this paper, we consider one-sided
subshifts. $(X, \sigma_X)$ is a subshift if $X$ is a closed
shift-invariant subset of $\{1,\cdots, k\}^{\N}$ for some $k\geq
1,$ where the shift $\sigma_X:X\rightarrow X$ is defined by
$\sigma_{X}(x)=x'$, for $x=(x_n)^{\infty}_{n=1}, x^{'}=
(x'_n)^{\infty}_{n=1}\in X, x'_{n}=x_{n+1}$ for all $n\in \N$.
Define a metric $d$  on $X$ by $d(x,x')={1}/{2^{k}}$ if
$x_i={x'}_i$ for all $1\leq i\leq k$ and $x_{k+1}\neq {x'}_{k+1},$
and $d(x,x')=0$ otherwise.  For each $n \in \N,$ denote by
$B_n(X)$ the set of all $n$-blocks that occur in points in $X.$
$x_1 \cdots x_n$ is an allowable word of length $n$ if $x_1\cdots
x_n\in B_n(X).$ If $x_1\cdots x_n \in B_n(X),$ then $[x_1 \cdots
x_n]$ is a cylinder in $X.$  Denote by $M(X,\sigma_X)$  the
collection of all $\sigma_X$-invariant Borel probability measures
on $X$ and by $Erg(X,\sigma_X)$ all ergodic members of
$M(X,\sigma_X)$.

Let $(X,\sigma_X)$ be a subshift and $f:X\rightarrow \R$  a
bounded Borel measurable function.
A $\sigma_X$-invariant Borel probability measure $\mu$ on $X$ is an {\em
equilibrium state }for $f$ if $h_{\mu}(\sigma_X)+\int f
d\mu=\sup_{\mu\in M(X,\sigma_X)}\{h_{\mu}(\sigma_X)+\int f
d\mu\}.$ Denote by $M_f(X,\sigma_X)$ the collection of all equilibrium states for $f$.
Let $\Phi=\{\phi_n\}^{\infty}_{n=1}$ be a sequence of real-valued continuous
functions on $(X,\sigma_X)$ for every $n\in \N$.  Barreira
\cite{B2006} and Mummert \cite{m2006} considered  almost additive
potentials. $\Phi=\{\phi_n\}^{\infty}_{n=1}$ is an {\em almost
additive potential} if for every $n,m\in \N$ and for every $x\in
X$ there exists a constant $C>0$ such that $-C+
\phi_{n}(x)+\phi_{m}({\sigma^{n}_X} x)\leq \phi_{n+m}(x)\leq C+
\phi_{n}(x)+\phi_{m}({\sigma^{n}_X} x).$ Let $\gamma_n(\Phi)=
\sup\{\vert \phi_n(x)-\phi_n(x')\vert: x_i={x'}_i \textnormal{ for all
} 1\leq i\leq n\}.$ Then $\Phi=\{\phi_n\}^{\infty}_{n=1}$ has {\em
bounded variation} if $\sup_{n\in \N} \gamma_n(\Phi)<\infty.$ The
notion of Gibbs measure for a continuous function  is
generalized.
\begin{definition}\cite{{B2006},{m2006}}\label{dGibbs}
Let $\Phi=\{\phi_n\}^{\infty}_{n=1}$ be an almost additive
potential on a shift of finite type $(X,\sigma_X)$. A Borel probability measure $\mu$ on $X$ is
a {\em Gibbs measure} for $\Phi$ if there exist $C>0$ and $P$ such
that
\begin{equation}\label{gp}
\frac{1}{C}<\frac{\mu([x_1 x_2 \cdots x_n])}{e^{-nP+\phi_n(x)}}<C
\end{equation}
for every $x \in X$ and $n\in \N.$
\end{definition}
The variational principle for almost additive potentials and
uniqueness of equilibrium states under some conditions were
separately studied by Barreira \cite{B2006} and Mummert
\cite{m2006}. We summarize the results that we shall need from
\cite{B2006} and \cite{m2006}.
\begin{theorem}\label{BM}(Special case of Theorems 1 and 5 in \cite{B2006} and Theorems 4 and 6 in \cite{m2006})
Let $(X,\sigma_X)$ be a topologically mixing shift of finite type
and $\Phi=\{\phi_n\}^{\infty}_{n=1}$ be an almost additive
potential on $(X,\sigma_X)$ with bounded variation. Then
\begin{equation}\label{vp}
P(\Phi)=\sup_{\mu\in
M(X,\sigma_X)}\{h_\mu(\sigma_X)+\lim_{n\rightarrow\infty}\frac{1}{n}\int\phi_n
d\mu\}
\end{equation}
where $P(\Phi)=\lim_{n \rightarrow
\infty}({1}/{n})\log (\sum_{i_1\cdots i_n \in B_n(X)}\sup e^{\phi_n(x)}),$ where
the supremum is taken over all $x\in [i_1\cdots i_n], n \in \N$. There exists a unique
measure that attains the maximum in (\ref{vp}). It is Gibbs and
mixing.
\end{theorem}

The variational principle for subadditive potentials was studied
by Cao, Feng, and Huang \cite{CFH}. Let
$\Phi=\{\phi_n\}^{\infty}_{n=1}$ be a sequence of real-valued
continuous functions on $(X, \sigma_X)$. Suppose that $\Phi$
satisfies the subadditivity condition, i.e., for every $n,m\in \N$
and $x \in X,$ $\phi_{n+m}(x)\leq \phi_{n}(x)+
\phi_{m}({\sigma^{n}_X} x)$. Then $\Phi$ is a {\em subadditive
potential}. We note that if $\phi\in C(X),$ defining $\phi_n(x)=
\sum^{n-1}_{k=0} \phi\circ \sigma_X^{k},$
$\Phi=\{\phi_n\}^{\infty}_{n=1}$ is subadditive. We define the
topological pressure of $\Phi$ using separated sets (see
\cite{CFH}). Let $\epsilon>0.$ A set $E$ is called an
$(n,\epsilon)$ separated subset of $X$ with respect to $\sigma_X$
if $\max_{0\leq i\leq
n-1}d(\sigma^{i}_{X}(x),\sigma^{i}_{X}(y))>\epsilon$ for all
$x,y\in E, x \neq y$. Define $P_n(\sigma_X, \Phi,
\epsilon)=\sup\{\sum_{x\in E} e^{\phi_n(x)}: E \textnormal{ is an
} (n,\epsilon)\textnormal{ separated subset of } X\}$ and
$P(\sigma_X, \Phi, \epsilon)=\limsup_{n\rightarrow \infty
}(1/n)\log P_n(\sigma_X, \Phi, \epsilon).$ Define the {\em
subadditive topological pressure of $\Phi$} with respect to
$\sigma_X$ by $P(\sigma_X, \Phi)=\lim_{\epsilon\rightarrow
0}P(\sigma_X, \Phi, \epsilon).$
\begin{theorem}\label{CP}(Special case of Theorem 1.1 in \cite{CFH})
Let $\Phi=\{\phi_n\}^{\infty}_{n=1}$ be a subadditive potential on
a subshift $(X,\sigma_X)$. Then
\begin{equation*}
P(\sigma_X, \Phi)=\sup_{\mu\in
M(X,\sigma_X)}\{h_{\mu}(\sigma_X)+\lim_{n \rightarrow
\infty}\frac{1}{n}\int\phi_n d\mu\}.
\end{equation*}
$P(\sigma_X, \Phi)=-\infty$ if and only if $\lim_{n\rightarrow\infty}(1/n)\int \phi_n d\mu=-\infty$ for all $\mu\in M(X,\sigma_X)$.
\end{theorem}
A $\sigma_X$-invariant Borel probability measure $\mu$ on $X$ is an {\em equilibrium state} for $\Phi$ if $h_{\mu}(\sigma_X)+ \lim_{n \rightarrow
\infty}(1/n)\int\phi_n d\mu =\sup_{\mu\in M(X, \sigma_X)}\{h_{\mu}(\sigma_X)+\lim_{n \rightarrow
\infty}(1/n)\int\phi_n d\mu\}.$ We denote
 the collection of all equilibrium states for $\Phi$ by $M_{\Phi}(X,\sigma_X)$. A Borel probability measure $\mu$ is a {\em Gibbs measure} for
 a subadditive potential $\Phi$ on $(X,\sigma_X)$ if (\ref{gp}) in Definition \ref{dGibbs} is satisfied. \label{defesp}

Next we summarize basic definitions. Let $(X, \sigma_X )$ and $(Y,
\sigma_Y )$ be subshifts and $\pi:(X,\sigma_X) \rightarrow
(Y,\sigma_Y)$ a factor map. If the $i$-th position of the image of
$x$ under $\pi$ depends only on $x_i,$ then $\pi$ is a one-block
factor map. Throughout the paper, we assume that $\pi$ is a
one-block factor map.   \label{En} For each $n\in \N$ and
$y=y_1\cdots y_n\cdots \in Y$, denote by $E_{n}(y)$ a set
consisting of exactly one point from each cylinder $[x_1\cdots
x_n]$ in $X$ such that $\pi ([x_1 \cdots x_n]) \subseteq [y_1
\cdots y_n].$ We note that for $y, y'\in Y$ with $y_i=y^{'}_{i}$
for all $1\leq i\leq n,$  the cardinality of $E_n(y)$ is equal to
that of $E_n(y').$ Denote by $\vert \pi^{-1}[y_1 \cdots y_n]\vert
$ the cardinality of the set $E_n(y).$ Denote by $D_n(y)$ a set
consisting of one point from each nonempty set $\pi^{-1}(y)\cap
[i_1 \cdots i_n]$ in $X$ and by $\vert D_n(y)\vert$ the
cardinality of the set $D_n(y)$. Using this notation, for $y=y_1
\cdots y_n \cdots \in Y, \vert D_{n}(y) \vert \leq \vert
\pi^{-1}[y_1\cdots y_n]\vert.$ If $X$ is an irreducible shift of
finite type, $\vert \pi^{-1}[y_1 \cdots y_n]\vert $ is the number
of blocks $x_1\cdots x_n$ of length $n$ in $X$ that are mapped to
the block $y_1 \cdots y_n$ in $Y.$

 Using the above notation, we now review some results on relative pressure. Relative pressure was studied by Ledrappier and Walters \cite{LW}.
\begin{theorem}\cite{LW}(Special case of the relative variational principle)
Let $(X, \sigma_X), (Y, \sigma_Y)$ be subshifts and $\pi:(X,
\sigma_X) \rightarrow (Y,\sigma_Y)$ a factor map. Let $f \in C(X)$
and $m\in M(Y,\sigma_Y).$ Then there exists a Borel measurable
function $P(\sigma_X, \pi, f):Y\rightarrow \R$ such that
$$\int P(\sigma_X, \pi,f)dm=\sup\{h_{\mu}(\sigma_X)-h_{m}(\sigma_Y)+\int f d\mu :\mu\in M(X,\sigma_X) \textnormal{ and }\pi\mu=m\}.$$
\end{theorem}
$\mu \in M(X, \sigma_X)$ is a {\em measure of maximal relative
entropy} over $m \in M(Y, \sigma_Y)$, if $\mu$ attains the maximum
in $\sup \{h_{\mu}(\sigma_X):\mu\in M(X,\sigma_X), \pi\mu=m \}.$
The properties of the Borel measurable function $P(\sigma_X, \pi,
f)$ were studied by Walters \cite{Wcom}.
\begin{theorem}\label{aboutP}(Special case of Theorem 4.6. in \cite{Wcom})
Let $(X, \sigma_X), (Y, \sigma_Y)$ be subshifts and
$\pi:(X,\sigma_X) \rightarrow (Y,\sigma_Y)$ a factor map. For each
 $y\in Y,$
$$P(\sigma_X, \pi, 0)(y)=\limsup_{n \rightarrow \infty}\frac{1}{n}\log \vert D_n(y) \vert. $$
\end{theorem}
It is not easy to calculate $\vert D_n(y)\vert$ in general. Some
other simpler function can be substituted for $P(\sigma_X, \pi,
0)$ in some cases.
\begin{theorem} \label{iPS} (Special case of Corollary in \cite{PS})
Let $(X,\sigma_X)$ be an irreducible shift of finite type, $(Y,
\sigma_Y)$ a subshift and $\pi:(X,\sigma_X) \rightarrow
(Y,\sigma_Y)$ factor map. For each $y\in Y,$
$$P(\sigma_X, \pi, 0)(y)=\limsup_{n\rightarrow\infty}\frac{1}{n}\log \vert
\pi^{-1}[y_1 \cdots y_n]\vert, $$ almost everywhere with respect
to every $\sigma_Y$-invariant Borel probability measure on $Y$.
\end{theorem}
\section{Existence of A Borel Measurable Saturated Compensation Function}\label{main1}

We first study the existence of a bounded Borel measurable saturated
compensation function for a factor map between subshifts.

\begin{lemma}\label{key1}
Let $(X,\sigma_X), (Y, \sigma_Y)$ be subshifts and $\pi:(X,
\sigma_X) \rightarrow (Y, \sigma_Y)$ a factor map. For $m\in Erg(Y, \sigma_Y)$, there exists $\mu\in M(X,
\sigma_X)$ such that $\pi \mu=m$ and
$h_{\mu}(\sigma_X)=\sup \{h_{\bar \mu}(\sigma_X) : \bar\mu\in
M(X,\sigma_X), \pi \bar\mu=m\}.$
\end{lemma}
 \begin{proof} By Corollary 3.2 \cite{Wcom}, for an  $m\in Erg(Y, \sigma_Y)$, there exists $\mu \in M(X,\sigma_X)$ such that $\pi \mu=m$ and $\mu$ is an equilibrium state of $\phi\circ \pi$ for some $\phi\in C(Y).$  By Proposition 6.1.1.\cite{S2}, such  $\mu$ is a measure of maximal relative entropy over $\pi\mu(=m).$
\end{proof}

\begin{theorem}\label{thm1}
Let $(X,\sigma_X), (Y, \sigma_Y)$ be subshifts and $\pi:(X,
\sigma_X) \rightarrow (Y, \sigma_Y)$ a factor map. Define
$F:Y\rightarrow \R$ by $F(y)=P(\sigma_X, \pi, 0)(y).$ Then
$F:Y\rightarrow \R$ is a bounded Borel measurable function and
$-F\circ\pi$ is a compensation function for
$(\sigma_X,\sigma_Y,\pi)$, i.e.,
\begin{align}
&\sup_{\mu \in M(X, \sigma_X)}\{h_{\mu}(\sigma_X)+\int
(-F\circ\pi+\phi\circ\pi)d \mu\} \label{gsumu}\\&=\sup_{m\in
M(Y,\sigma_Y)}\{h_{m}(\sigma_Y)+\int \phi dm\} \label{gsumd}
\end{align}
for all $\phi\in C(Y)$. Then $ M_{-F \circ \pi+\phi\circ\pi}(X,
\sigma_X)$ is nonempty. $\mu $ is an equilibrium state for $-F
\circ \pi+\phi\circ\pi$ if and only if (i) $\pi\mu$ is an
equilibrium state for $\phi$ and (ii) $\mu$ is a measure of maximal
relative entropy over $\pi\mu.$
\end{theorem}

\begin{proof}[Proof of equality (\ref{gsumd}) in Theorem \ref{thm1} ]
We first show that
\begin{equation}
\label{E1}
\sup_{\mu \in M(X,
\sigma_X)}\{h_{\mu}(\sigma_X)+\int (-F\circ\pi+\phi\circ\pi)d
\mu\}\leq \sup_{m\in M(Y, \sigma_Y)}\{h_{m}(\sigma_Y)+\int \phi
dm\} \textnormal{ for all }\phi\in C(Y).
\end{equation}
By the relative variational principle, there exists a Borel
measurable function $P(\sigma_{X}, \pi, 0):Y\rightarrow \R$  such
that for each $m\in M(Y, \sigma_Y)$,
$$\int P(\sigma_{X}, \pi, 0)dm =\sup \{h_{\mu}(\sigma_{X})-h_{m}(\sigma_{Y}): \mu \in M(X, \sigma_X) \text{ and } \pi \mu=m\}.$$
For $\mu\in M(X, \sigma_X),$
\begin{equation*}
\begin{split}
& h_{\mu}(\sigma_X)-\int
P(\sigma_{X}, \pi, 0)\circ\pi d\mu+\int \phi\circ\pi d \mu \\&=
h_{\mu}(\sigma_X)-\sup
\{h_{\tilde{\mu}}(\sigma_{X})-h_{\pi \mu}(\sigma_{Y}):
\tilde{\mu}\in M(X,\sigma_X), \pi {\tilde{\mu}}=\pi \mu\}+ \int
\phi \circ\pi d \mu \\&=
h_{\mu}(\sigma_X)-\sup \{h_{\tilde{\mu}}(\sigma_{X}):
\tilde {\mu}\in M(X, \sigma_X), \pi {\tilde{\mu}}=\pi \mu
\}+h_{\pi \mu}(\sigma_{Y})+ \int \phi \circ\pi d \mu \\& \leq
h_{\pi\mu}(\sigma_Y)+\int \phi  d\pi\mu \leq \sup_{m\in M(Y, \sigma_Y)}\{h_m(\sigma_Y)+\int\phi dm\}
\end{split}
\end{equation*}
where the last inequality follows from  $h_{\mu}(\sigma_X)\leq
\sup\{h_{\tilde{\mu}}(\sigma_{X}): \tilde{\mu}\in M(X,
\sigma_X),\pi {\tilde{\mu}}=\pi \mu\}.$ Taking the supremum over all $\mu\in M(X, \sigma_X)$, we obtain inequality (\ref{E1}).

For the reverse inequality, we use Lemma \ref{key1}. For
$m \in Erg(Y, \sigma_Y)$,  we
can find $\tilde{\mu}\in M(X, \sigma_X)$ such that $\pi
\tilde{\mu}=m$ and
$h_{\tilde{\mu}}(\sigma_X)=\sup\{h_{\bar{\mu}}(\sigma_X):
\bar{\mu}\in M(X, \sigma_X),\pi\bar{\mu}=m \}$. Hence, for $m \in Erg(Y, \sigma_Y)$,
\begin{equation*}
\begin{split}
&h_{m}(\sigma_Y)+\int\phi dm \\&=
h_{\tilde{\mu}}(\sigma_X)-\sup \{h_{\bar{\mu}}(\sigma_X): \bar
{\mu}\in M(X, \sigma_X),\pi {\bar{\mu}}=\pi \tilde {\mu}
\}+h_{\pi\tilde{\mu}}(\sigma_Y)+\int\phi \circ \pi d{\tilde{\mu}}
\\& \leq \sup_{ \mu \in M(X, \sigma_X)}\{h_{\mu}(\sigma_X)-\sup \{h_{\bar{\mu}}(\sigma_{X}): \bar{\mu}\in M(X, \sigma_X), \pi {\bar{\mu}}=\pi \mu \}+h_{\pi \mu}(\sigma_{Y})+ \int \phi\circ\pi d \mu\}\\
&=\sup_{ \mu \in M(X, \sigma_X)}\{h_{\mu}(\sigma_X)-\sup
\{h_{\bar{\mu}}(\sigma_{X})-h_{\pi \mu}(\sigma_{Y}) : \bar
{\mu}\in M(X, \sigma_X), \pi {\bar{\mu}}=\pi \mu \}+ \int
\phi\circ\pi d \mu\}\\& =\sup_{ \mu \in M(X,
\sigma_X)}\{h_{\mu}(\sigma_X)-\int P(\sigma_{X}, \pi,0)\circ \pi
d\mu + \int \phi\circ\pi d \mu\}.
\end{split}
\end{equation*}
Taking the supremum over all $m\in Erg(Y, \sigma_Y),$ we obtain the reverse inequality.
Thus, equality (\ref{gsumd}) holds. Theorem \ref{aboutP} implies that $P(\sigma_X, \pi,0)$ is bounded by $\log S$, where $S$ is the number of symbols in $X$.
\end{proof}
The rest of Theorem \ref{thm1} follows from the following three lemmas. In the lemmas, we assume that $(X,\sigma_X)$ and  $(Y, \sigma_Y)$ are
subshifts, $\pi:(X, \sigma_X) \rightarrow (Y, \sigma_Y)$ is a factor map and $F(y)=P(\sigma_X, \pi, 0)(y)$ for each $y\in Y$.
\begin{lemma}\label{c1}
Suppose ${m}$ is an ergodic equilibrium state
for $\phi\in C(Y).$ Let ${\mu}$ be a measure of maximal
relative entropy over ${m}$. Then ${\mu}$ attains the
maximum in (\ref{gsumu}). Therefore $M_{-F\circ
\pi+\phi\circ\pi}(X,\sigma_X)$ is nonempty.
\end{lemma}
\begin{proof}
If ${m}$ is an ergodic equilibrium state for
$\phi\in C(Y),$ then there exists a measure ${\mu}$ of maximal
relative entropy over ${m} $ (by Lemma \ref{key1}).
\begin{equation*}
\begin{split}
&h_{{\mu}}(\sigma_X)-\int P(\sigma_X, \pi, 0)\circ \pi d{{\mu}}+\int\phi\circ \pi d{{\mu}}\\
&=h_{{\mu}}(\sigma_X)-\sup \{h_{\bar {\mu}}(\sigma_X)-h_{\pi
{\mu}}(\sigma_Y) :\bar\mu\in M(X,\sigma_X), \pi {\bar\mu}=\pi
{\mu}\}+ \int\phi\circ\pi
d{{\mu}}\\&=h_{{\mu}}(\sigma_X)-\sup \{h_{\bar\mu}(\sigma_X):
\bar \mu\in M(X, \sigma_X), \pi {\bar \mu}=\pi {\mu}\}+h_{\pi
{\mu}}(\sigma_Y) +\int\phi\circ\pi d{{\mu}}\\&=h_{\pi
{\mu}}(\sigma_Y) +\int\phi d{{\pi\mu}} =\sup_{\nu\in
M(Y, \sigma_Y)}\{h_{\nu}(\sigma_Y)+\int \phi d\nu\}.
\end{split}
\end{equation*}
\end{proof}
\begin{lemma}\label{c1.5}
Suppose $m$ is an equilibrium state for  $\phi\in C(Y)$ and $\mu$
is a measure of maximal relative entropy over $m$. Then $\mu$ is
an equilibrium state for $-F\circ\pi+\phi\circ\pi.$
\end{lemma}
\begin{proof}
This is clear from the proof of Lemma \ref{c1}.
\end{proof}

\begin{lemma}\label{c2}
Suppose $\mu$ is an equilibrium state for  $-F\circ \pi+\phi\circ\pi, \phi\in C(Y).$
Then $\pi \mu$ is an equilibrium state for $\phi$ and $\mu$ is a measure of maximal relative entropy over
$\pi \mu$.
\end{lemma}
\begin{proof}
Let $\mu$ be an equilibrium state for  $-F\circ
\pi+\phi\circ\pi$. Assume that $\mu$ is not a measure of maximal
relative entropy over $\pi\mu$. Then
\begin{align}
 &\sup_{\bar {\mu} \in M(X, \sigma_X)}\{h_{\bar {\mu}}(\sigma_X)+\int (-F\circ\pi+\phi\circ\pi)d \bar{\mu}\}\\
 &=h_{\mu}(\sigma_X)+\int(-F\circ\pi+\phi\circ\pi)d\mu\\
 &= h_{\mu}(\sigma_X)-\sup\{h_{\tilde{\mu}}(\sigma_X):\tilde{\mu}\in M(X, \sigma_X), \pi \tilde{\mu}=\pi \mu\}
+ h_{\pi\mu}(\sigma_Y)+\int \phi\circ \pi d\mu\\
& < h_{\pi \mu}(\sigma_Y)+\int \phi d \pi \mu \label{<}\\&\leq
\sup_{m\in M(Y,
 \sigma_Y)}\{h_m(\sigma_Y)+\int \phi dm\} .
\end{align}
This is a contradiction to equality (\ref{gsumd}). Now we show
that $\pi\mu$ is an equilibrium state for $\phi.$ Since
$\mu$ is a measure of maximal relative entropy over $\pi\mu$,
replacing $<$ by $=$ in (\ref{<}) and using equality
(\ref{gsumd}), we obtain
$$h_{\pi \mu}(\sigma_Y)+\int \phi d \pi \mu =\sup_{m\in M(Y, \sigma_Y)}\{h_m(\sigma_Y)+\int \phi dm\}$$
Therefore, $\pi\mu$ is an equilibrium state for $\phi.$
\end{proof}

Next we will show that Theorem \ref{thm1} is still valid when we
replace $\phi\in C(Y)$ in (\ref{gsumu}) and (\ref{gsumd}) by
subadditive potentials on $(Y,\sigma_Y).$

We recently found the preprint \cite{FH} by Feng and Huang and the
upper semi-continuity in the proof of Proposition \ref{basic1}
(\ref{b2}) is shown in a more general setting in Proposition
3.1.(2) \cite{FH} for a larger class of potentials that contain
subadditive potentials.
\begin{proposition}\label{basic1}
Let $\Phi=\{\phi_n\}^{\infty}_{n=1}$ be a subadditive potential on
a subshift $(Y, \sigma_Y).$ Then
\begin{enumerate}
\item $M_{\Phi}(Y, \sigma_Y)$ is convex. \label{b1}\\
\item $M_{\Phi}(Y, \sigma_Y)$ is nonempty and compact.\label{b2}\\
\item The extreme points of $M_{\Phi}(Y, \sigma_Y)$ are precisely the ergodic members of $M_{\Phi}(Y, \sigma_Y).$\label{b3}\\
\item $\sup_{\mu\in M(Y, \sigma_Y)}\{h_{\mu}(\sigma_Y)+\lim_{n \rightarrow \infty}(1/n)\int\phi_{n} d\mu\}\label{b4}\\=
\sup_{\mu\in Erg(Y, \sigma_Y)}\{h_{\mu}(\sigma_Y)+\lim_{n\rightarrow \infty}(1/n)\int\phi_n d\mu\}.$\\
\item $M_{\Phi}(Y, \sigma_Y)$ contains an ergodic measure. \label{b5}
\end{enumerate}
\end{proposition}
\begin{proof} By Theorem \ref{CP}, if $P(\sigma_Y, \Phi)=-\infty$, then $M_{\Phi}(Y, \sigma_Y)=M(Y, \sigma_Y)$ and we obtain the results.
Therefore, we consider the case when $P(\sigma_Y, \Phi)\neq -\infty.$ (\ref{b1}) is clear. For (\ref{b2}), we first notice that Subadditive Ergodic
Theorem implies that $h_{\mu}(\sigma_Y)+\lim_{n \rightarrow \infty}(1/n)\int\phi_{n} d\mu$ takes a value in $[-\infty, \infty)$ for each
$\mu\in M(Y, \sigma_Y)$. Let $\{\mu_{m}\}^{\infty}_{m=1}$ converge to $ \mu \in M(Y, \sigma_Y)$ in the weak* topology. Since the entropy map is
upper semi-continuous,  in order to show that $M_{\Phi}(Y, \sigma_Y)$ is nonempty, it is enough to prove
$$\limsup_{m\rightarrow \infty}\lim_{n\rightarrow\infty}\frac{1}{n}\int\phi_{n}d\mu_{m}\leq \lim_{n\rightarrow\infty}\frac{1}{n}\int\phi_{n}d\mu.$$
Using subadditivity, we have for $\mu\in M(Y, \sigma_Y)$,
$$\lim_{n\rightarrow\infty}\frac{1}{n}\int\phi_{n}d\mu\leq \frac{1}{n}\int\phi_{n}d\mu \text { for all } n \in \N.$$
Let $\mu_{m}$ be fixed. Then
$$\lim_{n\rightarrow\infty}\frac{1}{n}\int\phi_{n}d\mu_m\leq \frac{1}{n}\int\phi_{n}d\mu_m. $$
Let $n$ be fixed and $m\rightarrow \infty$. Then
\begin{equation*}
\begin{split}
\limsup_{m\rightarrow\infty}\lim_{n\rightarrow\infty}\frac{1}{n}\int\phi_{n}d\mu_m
& \leq
\limsup_{m\rightarrow\infty}\frac{1}{n}\int\phi_{n}d\mu_m\\& =
\frac{1}{n}\int\phi_{n}d\mu.
\end{split}
\end{equation*}
Letting $n\rightarrow \infty,$
$$\limsup_{m\rightarrow\infty}\lim_{n\rightarrow\infty}\frac{1}{n}\int\phi_{n}d\mu_m\leq \lim_{n\rightarrow\infty}\frac{1}{n}\int\phi_{n}d\mu.$$
Compactness follows by using the upper semi-continuity of the map $\mu\rightarrow h_{\mu}(\sigma_Y)+$\\
\noindent $\lim_{n\rightarrow \infty}(1/n)\int\phi_n d\mu$. For
(\ref{b3}), (\ref{b4}) and (\ref{b5}), noticing that $P(\sigma_Y,
\Phi)<\infty$ by (\ref{b2}), we make the standard arguments as in
the proofs of Theorem 8.7 and Corollary 8.6.1 (i) in \cite{Wtext},
combining with Proposition 3.1.(3) \cite{FH}.
\end{proof}
Let $(X,\sigma_X)$ be a subshift. Let $f:X\rightarrow \R$ be a
bounded Borel measurable function and
$\Phi=\{\phi_n\}^{\infty}_{n=1}$ a subadditive
potential on $(X,\sigma_X).$ A $\sigma_X$-invariant Borel
probability measure $\mu$ on $X$ is an {\em equilibrium state} for
$f +\Phi,$ if $h_{\mu}(\sigma_X)+\int f d\mu +\lim_{n\rightarrow
\infty}(1/n)\int \phi_n d \mu=\sup_{\mu \in M(X,
\sigma_X)}\{h_{\mu}(\sigma_X)+\int f d\mu+
\lim_{n\rightarrow\infty}(1/n)\int \phi_{n} d \mu \}.$ Denote by
$M_{f+\Phi}(X, \sigma_X)$ the collection of all equilibrium states
for $f+\Phi.$ It is easy to see that $\sup_{\mu \in M(X,
\sigma_X)}\{h_{\mu}(\sigma_X)+\int f d\mu+
\lim_{n\rightarrow\infty}(1/n)\int \phi_{n} d \mu \}=-\infty$ if and only if  $\lim_{n\rightarrow\infty}(1/n)\int \phi_{n} d \mu=-\infty$ for all $\mu\in M(X,\sigma_X)$.
\begin{proposition}\label{basic2}
Let $(X,\sigma_X), (Y, \sigma_Y)$ be subshifts  and $\pi:(X,
\sigma_X) \rightarrow (Y,\sigma_Y)$  a factor map. Let
$\Phi=\{\phi_n\}^{\infty}_{n=1}$ be a subadditive potential on
$(Y, \sigma_Y)$. Define $\Phi
\circ\pi=\{\phi_n\circ\pi\}^{\infty}_{n=1}$ and $F:Y\rightarrow
\R$ as in Theorem \ref{thm1}. Then
\begin{enumerate}
\item $M_{-F\circ\pi+\Phi\circ\pi}(X, \sigma_X)$ is convex.\label{b21}\\
\item $M_{-F\circ\pi+\Phi\circ\pi}(X, \sigma_X)$ is nonempty. \label{b22}\\
\item The extreme points of $M_{-F\circ\pi+\Phi\circ\pi}(X,\sigma_X)$ are precisely the ergodic members of    $M_{-F\circ\pi+\Phi\circ\pi}(X, \sigma_X).$ \label{b23}\\
\end{enumerate}
\end{proposition}
\begin{proof} We consider the case when
$\sup_{\mu\in M(X, \sigma_X)}\{h_{\mu}(\sigma_X)-\int F\circ\pi
d\mu +\lim_{n\rightarrow \infty}(1/n)\int\phi_n \circ \pi
d\mu\}\neq -\infty$. (\ref{b21}) is clear. For (\ref{b22}), we
claim that equality (\ref{gsumd}) in Theorem \ref{thm1} holds when
we replace $\phi$ by $\Phi$. To see this, replacing $\phi$ by
$\Phi$, we take the supremum over all $\mu\in M(X, \sigma_X)$ such
that $\lim_{n\rightarrow \infty}(1/n)\int\phi_n \circ \pi d\mu
\neq -\infty$ in the proof of inequality (\ref{E1}). This gives
inequality (\ref{E1}) with $\phi$ replaced by $\Phi$. For the
reverse inequality, we make the same proof as in the proof of
equality (\ref{gsumd}) in Theorem \ref{thm1}, taking the supremum
over all $m\in Erg(Y, \sigma_Y)$ such that $\lim_{n\rightarrow
\infty}(1/n)\int\phi_n dm \neq -\infty$ and using Proposition
\ref{basic1} (\ref{b4}). Now the claim is proved. Using
Proposition \ref{basic1} (\ref{b5}), Lemma \ref{c1} holds when we
replace $\phi$ by $\Phi$. Therefore, we obtain (\ref{b22}).  For
(\ref{b23}), since (\ref{b2}) implies $\sup_{\mu\in M(X,
\sigma_X)}\{h_{\mu}(\sigma_X)-\int F\circ\pi d\mu
+\lim_{n\rightarrow \infty}(1/n)\int\phi_n \circ \pi
d\mu\}<\infty$,  we use \cite{Wtext} as in the proof of
Proposition \ref{basic1}.
\end{proof}
\begin{theorem}\label{thm2}
Let $(X,\sigma_X), (Y, \sigma_Y)$ be subshifts  and $\pi:(X,
\sigma_X) \rightarrow (Y,\sigma_Y)$ a factor map. Let
$\Phi=\{\phi_n\}^{\infty}_{n=1}$ be a subadditive potential on
$(Y, \sigma_Y)$. Define $\Phi
\circ\pi=\{\phi_n\circ\pi\}^{\infty}_{n=1}$ and $F:Y\rightarrow
\R$ as in Theorem \ref{thm1}. Then
\begin{align}
& \sup_{\mu \in M(X, \sigma_X)}\{h_{\mu}(\sigma_X)-\int F\circ\pi d\mu+ \lim_{n\rightarrow\infty}\frac{1}{n}\int \phi_{n}\circ\pi d \mu\} \label{subpx}\\
&=\sup_{m\in
M(Y,\sigma_Y)}\{h_{m}(\sigma_Y)+\lim_{n\rightarrow\infty}\frac{1}{n}\int
\phi_{n} dm\}, \label{subpy}
\end{align}
for all $\Phi$.  Then $M_{-F\circ \pi+\Phi\circ\pi}(X,\sigma_X)$
is nonempty. If $P(\sigma_Y, \Phi)\neq -\infty$, then $\mu$ is an
equilibrium state for $-F\circ \pi+\Phi\circ\pi$ if and only if
(i) $\pi\mu$ is an equilibrium state for $\Phi$ and (ii) $\mu$ is
a measure of maximal relative entropy over $\pi \mu$.
\end{theorem}
\begin{proof}
If $P(\sigma_Y, \Phi)=-\infty$, then we clearly have the results.
Suppose $P(\sigma_Y, \Phi)\neq -\infty$. Equality (\ref{subpy}) is
shown in the first part of the proof of Proposition \ref{basic2}
(\ref{b22}).
 For the if and only if part, we use Lemmas  \ref{c1.5} and
\ref{c2} with $\phi$ replaced by $\Phi$.
\end{proof}
\begin{corollary}\label{nice}
 Suppose $(X,\sigma_X)$ is an irreducible shift of finite type and $(Y, \sigma_Y)$ is a subshift.
 Then we can replace $\vert D_n(y)\vert $ of $F$ (see Theorem \ref{aboutP}) by $\vert \pi^{-1}[y_1\cdots y_n]\vert$ in Theorems \ref{thm1} and \ref{thm2}.
\end{corollary}
\begin{proof} Apply Theorem \ref{iPS}.
\end{proof}
\begin{proposition}
Under the assumptions of Proposition \ref{basic2}, $M_{-F\circ\pi +\Phi\circ\pi}(X,\sigma_X)$ is compact if $M_{\Phi}(Y, \sigma_Y)$ consists of one point.
\end{proposition}
\begin{proof}
If $P(\sigma_Y, \Phi)=-\infty$, then $M_{-F\circ\pi+\Phi\circ\pi}(X, \sigma_X)=M(X,\sigma_X)$ and so we obtain the result.  If $P(\sigma_Y, \Phi)\neq -\infty$, then Proposition \ref{basic1} (\ref{b2}) implies that  $P(\sigma_Y, \Phi)< \infty$.
Let $\mu_n\in M_{-F\circ\pi+\Phi\circ\pi}(X, \sigma_X)$ for all $n\in \N.$  Since
$\{\mu_n\}_{n=1}^{\infty}$ is a sequence in $M(X, \sigma_X)$, there exists a subsequence $\{\mu_{n_k}\}_{k=1}^{\infty}$ that converges to $\mu\in M(X,\sigma_X).$ We show that $\mu\in M_{-F\circ \pi +\Phi\circ\pi}(X,\sigma_X).$
Clearly $\{\pi \mu_{n_k}\}^{\infty}_{k=1}$ converges to $\pi\mu$ in the weak* topology. By Theorem \ref{thm2}, $\pi\mu_{n_k}\in M_{\Phi}(Y, \sigma_Y)$ for all $k\in \N$ and so
\begin{equation*}
\begin{split}
P(\sigma_Y, \Phi)&=\limsup_{k\rightarrow\infty}(h_{\pi \mu_{n_k}}(\sigma_Y)+\lim_{m\rightarrow \infty}\frac{1}{m}\int\phi_m  d\pi\mu_{n_k})\\
& \leq h_{\pi \mu}(\sigma_Y)+\lim_{m\rightarrow\infty
}\frac{1}{m}\int\phi_m  d\pi\mu\leq P(\sigma_Y, \Phi).
\end{split}
\end{equation*}
Thus $\pi\mu\in M_{\Phi}(Y, \sigma_Y).$ Now we show that $\mu$ is
a measure of maximal relative entropy over $\pi \mu.$ Since by Theorem \ref{thm2} $\mu_{n_k}$ is a measure of
maximal relative entropy over $\pi\mu_{n_k},$
\begin{equation*}
\begin{split}
& \limsup_{k\rightarrow\infty }(h_{\pi \mu_{n_k}}(\sigma_Y)+ \int P(\sigma_X, \pi, 0) d \pi \mu_{n_k})\leq h_{\mu}(\sigma_X)\\
& \leq h_{\pi \mu}(\sigma_Y)+ \int P(\sigma_X, \pi, 0) d \pi \mu.
\end{split}
\end{equation*}
 Since $M_{\Phi}(Y, \sigma_Y)$ consists of one point, if $M_{\Phi}(Y, \sigma_Y)=\{m\},$  then $\pi \mu_{n_k}=m$ for all $k \in \N.$ Therefore, $h_\mu(\sigma_X)=h_{\pi \mu}(\sigma_Y)+\int P(\sigma_X, \pi, 0)d\pi\mu.$ Now we apply Theorem \ref{thm2} to obtain $\mu\in M_{-F\circ \pi +\Phi\circ\pi}(X,\sigma_X).$
\end{proof}
In \cite{S2}, the measures that maximize the weighted entropy
functionals $\phi_{\alpha}$ were studied for the case when there
is a {\em continuous} saturated compensation function between
subshifts. We extend the results to arbitrary subshifts without
the existence of a {\em continuous} saturated compensation
function. Applying Theorem \ref{thm2}, we first study the case
when $(X,\sigma_X)$ is an irreducible shift of finite type.
\begin{proposition}\label{sftwentropy}
Let $(X,\sigma_X)$ be an irreducible shift of finite type, $(Y, \sigma_Y)$ a subshift and
$\pi:(X,\sigma_X) \rightarrow (Y,\sigma_Y)$ a factor map. Let
$\alpha>0.$  For each $n\in  \N,$ define
${\phi_n}:Y\rightarrow \R$ by ${\phi_n}(y)= \log \vert
\pi^{-1}[y_1 \cdots y_n] \vert^{1/(\alpha+1)}$ and let
$\Phi=\{\phi_n\}^{\infty}_{n=1}$. Then
\begin{align}
& \sup_{\mu \in M (X,\sigma_X)}\{h_{\mu}(\sigma_X)+\alpha h_{\pi \mu}(\sigma_Y)\} \label{0}\\
&=(\alpha+1)\sup_{m \in M(Y, \sigma_Y)}\{h_{m}(\sigma_Y)+\lim_{n\rightarrow \infty}\frac{1}{n}\int\phi_n dm\} \label{sftmd}\\&=
(\alpha+1) \sup_{\mu \in M(X, \sigma_X)}\{h_{\mu}(\sigma_X)-\lim_{n\rightarrow \infty}\frac{1}{n}\int\alpha \phi_n \circ \pi d\mu\}.\label{sftmu}
\end{align}
There exists a $\sigma_X$-invariant ergodic Borel probability measure on $X$ that attains the maximum in
(\ref{0}). Let $K_{\alpha}$ be the collection of measures that
attain the maximum in (\ref{0}). Then $\mu\in K_{\alpha}$ if and
only if (i) $\pi\mu$ is an equilibrium state for $\Phi$ and (ii) $\mu$ is a measure of maximal relative entropy over $\pi \mu.$
\end{proposition}
\begin{proof}
It was shown in \cite{S2} (Proposition 3.1.2) that a $\sigma_X$-invariant ergodic Borel probability measure on $X$ that attains
the maximum in (\ref{0}) exists for $(X, \sigma_X),(Y, \sigma_Y)$ subshifts. Clearly, if $\bar\mu \in K_\alpha,$ then $\bar
\mu$ attains the maximum in $\sup \{h_{\mu}(\sigma_{X})-h_{\pi
\bar \mu}(\sigma_{Y}):\mu\in M(X, \sigma_X), \pi{\mu}=\pi \bar
\mu\}.$ Define $F:Y\rightarrow \R$ as in Theorem \ref{thm1}. Then
\begin{equation*}
\begin{split}
&\sup_{\mu \in M(X, \sigma_X)}\{h_{\mu}(\sigma_X)+\alpha
h_{\pi\mu}(\sigma_Y)\}=
\sup_{\mu\in M(X, \sigma_X)}\{h_{\mu}(\sigma_X)-h_{\pi\mu}(\sigma_Y)+(1+\alpha)h_{\pi \mu}(\sigma_Y)\}\\
&=\sup_{\mu \in M(X, \sigma_X)}\{\sup \{h_{\tilde{\mu}}(\sigma_{X})-h_{\pi \mu}(\sigma_{Y}): \tilde{\mu}\in M(X, \sigma_X), \pi {\tilde{\mu}}=\pi \mu \}+ (1+\alpha)h_{\pi \mu}(\sigma_Y)\}\\
&=\sup_{\mu \in
M(X,\sigma_X)}\{(\alpha+1)h_{\pi\mu}(\sigma_X)+\int F d\pi\mu\}\\
&=(\alpha+1)\sup_{m \in
M(Y,\sigma_Y)}\{h_{m}(\sigma_Y)+\frac{1}{\alpha+1}\int F dm\}.
\end{split}
\end{equation*}
Clearly $\Phi$ is a subadditive potential on $(Y, \sigma_Y).$ Therefore, using
Theorem \ref{iPS} and  Kingman's Subadditive
Ergodic Theorem (Theorem 10.1 in \cite{Wtext}), we get equality (\ref{sftmd}).
For equality (\ref{sftmu}), we apply Theorem \ref{thm2} (set $\phi_n(y)=\log \vert
\pi^{-1}[y_1 \cdots y_n] \vert^{1/(\alpha+1)})$ and Corollary \ref{nice}. Now suppose $\mu\in K_{\alpha}.$ Since $\mu$ is a measure of maximal relative entropy over $\pi\mu$, we conclude that $\pi\mu$ is an equilibrium state for $\Phi$ by using the proof of (\ref{sftmd}).
  Next suppose that $m$ is an equilibrium state for $\Phi$ and $\mu$ is a measure of maximal relative entropy over $m$. Then
 \begin{equation*}
 \begin{split}
 &(\alpha+1)\big (h_m(\sigma_Y)+\lim_{n\rightarrow \infty}\frac{1}{n}\int  \phi_n  dm \big)
 =(\alpha+1)h_{\pi\mu}(\sigma_Y)+  \lim_{n\rightarrow \infty}\frac{1}{n}\int (\alpha+1) \phi_n  d\pi\mu  \\
 &=\sup \{h_{\tilde{\mu}}(\sigma_X)-h_{\pi\mu}(\sigma_{Y}):\tilde{\mu}\in M(X, \sigma_X), \pi\tilde{\mu}=\pi\mu\}+(\alpha+1)h_{\pi\mu}(\sigma_Y)\\
 &=h_{\mu}(\sigma_X)+\alpha h_{\pi\mu}(\sigma_Y)
 \end{split}
 \end{equation*}
 Using equality (\ref{sftmd}), we conclude that $\mu\in K_{\alpha}$.
\end{proof}
\begin{theorem}\label{wentropy}
Let $(X,\sigma_X),(Y, \sigma_Y)$ be subshifts and
$\pi:(X,\sigma_X) \rightarrow (Y,\sigma_Y)$ a factor map. Let
$\alpha>0$ and $F$ be defined as in Theorem \ref{thm1}. Then
\begin{align}
& \sup_{\mu \in M (X,\sigma_X)}\{h_{\mu}(\sigma_X)+\alpha h_{\pi \mu}(\sigma_Y)\} \label{0.5}\\&=(\alpha+1)\sup_{m \in M(Y, \sigma_Y)}\{h_{m}(\sigma_Y)+\frac{1}{\alpha+1}\int Fdm\} \label{eq1}\\
&=(\alpha+1)\sup_{\mu\in
M(X,\sigma_X)}\{h_{\mu}(\sigma_X)-\frac{\alpha}{\alpha+1}\int F
\circ \pi d \mu\}\label{eq2}
\end{align}
There exists a $\sigma_X$-invariant ergodic Borel probability measure that attains the maximum in
(\ref{0.5}). Let $K_{\alpha}$ be the collection of measures that
attain the maximum in (\ref{0.5}). Then $K_{\alpha}=M_{-(\alpha/(\alpha+1))F}(X,\sigma_X)$. $\mu\in K_{\alpha}$ if and
only if (i) $\pi\mu$ is an equilibrium state for $(1/(\alpha+1))
F$ and (ii) $\mu$ is a measure of maximal relative entropy over $\pi \mu.$
\end{theorem}
\begin{remark}
Walters (Theorem 3.4. \cite{Wcom}) showed that if there is a
saturated compensation function $G \circ \pi,$ $G \in C(Y)$, then
$\int G dm=-\int F dm$ for all $m\in M(Y,\sigma_Y)$. Therefore
Theorem \ref{wentropy}  extends the results of Shin (Theorem 1.1
\cite{S3}).
\end{remark}
\begin{proof} We get equality (\ref{eq1}) by the same proof as in Proposition \ref{sftwentropy}.
We notice that $M_{(1/(\alpha+1))F}(Y, \sigma_Y)$ contains an
ergodic measure. If $\mu$ is an ergodic member of $K_{\alpha}$,
then $\pi\mu$ is an ergodic member of $
M_{(1/(\alpha+1))F}(Y,\sigma_Y)$ by the proof of equality
(\ref{sftmd}) in Proposition \ref{sftwentropy}. Hence equality
(\ref{eq2}) follows from the proof of equality (\ref{gsumd}) in
Theorem \ref{thm1} (with $\phi=(1/(\alpha+1))F$). For the if and
only if part, we use the proof of Proposition \ref{sftwentropy},
replacing $\Phi$ by $(1/(\alpha+1))F$. We observe that we cannot
apply Theorem \ref{thm2} in order to show
$K_{\alpha}=M_{-(\alpha/(\alpha+1))F}(X,\sigma_X)$ because
$\Psi=\{\log \vert D_n(y)\vert\}^{\infty}_{n=1}$ is not always a
sequence of continuous functions on $(Y,\sigma_Y)$. Suppose
$\mu\in K_{\alpha}.$ Since $\pi\mu$ is an equilibrium state for
$(1/(\alpha+1)) F$ and $\mu$ is a measure of maximal relative
entropy over $\pi \mu,$ we conclude that $\mu\in
M_{-(\alpha/(\alpha+1))F}(X,\sigma_X)$ by using the proof of Lemma
\ref{c1.5} (with $\phi=(1/(\alpha+1))F$) and equality (\ref{eq2}).
Now suppose that $\mu\in M_{-(\alpha/(\alpha+1))F}(X,\sigma_X)$
(such a $\mu$ exists from the above). Since Lemma \ref{c2} holds
for $\phi=(1/(\alpha+1))F$, we obtain $\mu\in K_{\alpha}$.
\end{proof}
\section{Applications---dimensions of nonconformal expanding maps}\label{application}
We will apply the results from section \ref{main1} to find the
Hausdorff dimension of a compact invariant set of a nonconformal
expanding map and measures of full dimension. In particular, we
consider a compact invariant set of a nonconformal expanding map
of the torus represented by a diagonal matrix. For more details
and background material on this section, see \cite{Yayama}. We now
review the definitions of an NC carpet and SFT-NC carpet. Fix two
positive integers $l$ and $m$, $l>m\geq 2.$ Throughout this paper,
we let $T$ be the endomorphism of the torus ${\T}^{2}=
{\R}^{2}/{\Z}^{2}$ given by $T(x,y)=(lx \textnormal{ mod }1, my
\textnormal{ mod }1).$ Let
$$\p=\{ [\frac{i}{l}, \frac {i+1}{l}]\times[\frac{j}{m},
\frac{j+1}{m}]: 0\leq i \leq l-1, 0\leq j\leq m-1\}$$ be the
natural Markov partition for $T.$ Label $[\frac{i}{l}, \frac
{i+1}{l}]\times[\frac{j}{m}, \frac{j+1}{m}], 0\leq i\leq l-1,
0\leq j\leq m-1$, by the symbol $(i,j)$. Define $(\Sigma^{+}_{lm},
\sigma)$ to be the full shift on these $lm$ symbols. Consider the
coding map $\chi: \Sigma ^{+}_{lm}\rightarrow {\T}^{2},$ defined by
\label{codingmap}
\begin{displaymath}
\chi(\{(x_k,
y_k)\}_{k=1}^{\infty})=(\sum_{k=1}^{\infty}\frac{x_k}{l^k},
\sum_{k=1}^ {\infty} \frac{y_k}{m^k}).
\end{displaymath}
\label{symrep} Let $R=\{(a_1,b_1), (a_2,b_2) \cdots, (a_r,b_r)\}$
be a subalphabet of the symbols of $\p.$ The nonconformal carpet
(\em NC carpet\em) $K(T,R)$ is defined by \label{gscarpet}
\begin{equation*}K(T,R)=\{( \sum_{k=1}^{\infty}\frac{x_k}{l^k},
\sum_{k=1}^{\infty}\frac{y_k}{m^k}): (x_k, y_k)\in R \text{ for
all } k \in \N\}.
\end{equation*}
It is a compact $T$-invariant subset of the torus. Denote by $A$ a
transition matrix among the members of $R$,
 so that $A$ is an $r
\times r$ matrix with entries 0 or 1. The \em SFT-NC carpet
$K(T,R,A)$ \em is defined by \label{dsc}
\begin{equation*}
K(T,R,A)=\{(\sum_{k=1}^{\infty}\frac{x_k}{l^{k}},
\sum_{k=1}^{\infty} \frac{y_k}{m^{k}}): (x_k, y_k)\in R,
A_{(x_k,y_k)(x_{k+1}, y_{k+1})}=1 \textnormal{ for all } k\in \N
\}.
\end{equation*}
Now let $(X,\sigma_X)$ be the shift of finite type with the transition matrix $A$ as above. Let $\pi:X\rightarrow Y$ be the factor map determined by
the one-block map $\pi((a_k,b_k))=b_k$ and let $Y=\pi(X)$.   In \cite{Yayama},
the Hausdorff dimension for $K(T,R,A)$ was
studied. If there exists a saturated compensation function
$G\circ\pi,G\in C(Y),$ then the Hausdorff dimension is
given by
\begin{equation*}
\dim_{H}K(T, R, A)= \frac{1}{\log m}\sup_{\mu \in M(X, \sigma_X)}\{h_{\mu}(\sigma_X)+\frac{\alpha}{\alpha+1}\int G\circ\pi d\mu\},
\end{equation*}
and the $T$-invariant ergodic measures of full dimension are the
ergodic equilibrium states for $(\alpha/(\alpha+1))G\circ\pi.$
This formula extends the formula for the Hausdorff dimension of
$K(T,R)$ given by McMullen \cite{Mc}.

Given a shift of finite type $(X, \sigma_X)$, a subshift $(Y,
\sigma_Y)$ and factor map $\pi:(X,\sigma_X)\rightarrow (Y,
\sigma_Y),$ we can construct an SFT-NC carpet by defining $T$ and
choosing $R$ appropriately. The construction is not unique. We
call such a carpet an SFT-NC carpet corresponding to $(X, Y,\pi).$

We will extend the formula above to the general case by
considering a compact $T$-invariant set whose symbolic
representation is a subshift. Let $S\subset \Sigma^{+}_{r} \subset
\Sigma^{+}_{lm} $ be a subshift on the members of $R$. Then
$\chi(S)$ is a {\em subshift NC carpet}. Given any subshifts $(X,
\sigma_X), (Y, \sigma_Y)$ and a factor map
$\pi:(X,\sigma_X)\rightarrow (Y, \sigma_Y),$ we define a
subshift-NC carpet corresponding to  $(X, Y, \pi)$ in the same
manner as we did for an SFT-NC carpet. For the symbolic
representation of such a carpet, there does not always exist a
saturated compensation function. In Theorem \ref{application1}, we
will give a formula for the Hausdorff dimension of a subshift-NC
carpet and characterize the invariant ergodic measures of full
dimension as the ergodic equilibrium states of a bounded Borel
measurable function.

Recall that for $y\in Y, \vert \pi^{-1}[y_1\cdots y_{n}]\vert
$ denotes the cardinality of $E_{n}(y)$ (see page \pageref{En}) and, in particular, it is the number of blocks of  $x_1\cdots x_n$ of length $n$ in $X$ that are mapped to
the block $y_1 \cdots y_n$ in $Y$ if $X$ is an irreducible shift of
finite type. \label{defespp} Let  $\Phi=\{\phi_n\}_{n=1}^{\infty}$ be a subadditive potential on $(X,\sigma_X).$ Then $-\Phi=\{-\phi_n\}_{n=1}^{\infty}$ is a superadditive potential on $(X,\sigma_X)$. A measure $\mu\in M(X,\sigma_X)$ is an {\em equilibrium state} for $-\Phi$ if $h_{\mu}(\sigma_X)-\lim_{n\rightarrow\infty}(1/n)\int\phi_n d\mu=\sup_{\mu \in M(X, \sigma_X)}\{h_{\mu}(\sigma_X)-\lim_{n\rightarrow \infty}(1/n)\int\phi_n d\mu\}.$
We note that $\sup_{\mu \in M(X, \sigma_X)}\{h_{\mu}(\sigma_X)-\lim_{n\rightarrow \infty}(1/n)\int\phi_n d\mu\}$ takes a value in $(-\infty, \infty]$.
\begin{theorem}\label{HDDFORSFT}
Let $(X, \sigma_X)$  be an irreducible shift of finite type, $(Y,
\sigma_Y)$ a subshift, $\pi:(X,\sigma_X) \rightarrow (Y,\sigma_Y)$
factor map.  For $l,m\in \N, l>m\geq 2,$ set $\alpha=\log
_{m}l-1.$ Define $\phi_n:Y\rightarrow \R$ for each  $n\in \N$ by
$\phi_n(y)= \log \vert \pi^{-1} [y_1 \cdots y_n]
\vert^{1/(\alpha+1)}.$ Let $\Phi=\{\phi_n\}^{\infty}_{n=1}$ and
$-\alpha \Phi \circ\pi=\{-\alpha \phi_n\circ\pi\}^{\infty}_{n=1}.$
Then the Hausdorff dimension of an SFT-NC carpet $K$ corresponding
to $(X,Y,\pi)$ is given by
\begin{align}
\dim_{H}K &=\frac{1}{\log m}\sup_{\nu \in M(Y, \sigma_Y)}\{h_{\nu}(\sigma_Y)+\lim_{n\rightarrow\infty}\frac{1}{n}\int \phi_n d \nu\} \label{HDformula1}\\&=\frac{1}{\log m}\sup_{\mu \in M(X,\sigma_X)}\{h_{\mu}(\sigma_X)-\lim_{n\rightarrow\infty}\frac{1}{n}\int \alpha \phi_n \circ\pi d\mu\} \label{HDformula2}\\
&= \frac{1}{\log m}\limsup_{n\rightarrow\infty} \frac{1}{n} \log
(\sum_{y_1 \cdots y_n \in B_n(Y)}  \vert \pi^{-1}[y_1\cdots
y_n]\vert^{1/(\alpha+1)}).\label{HDformula3}
\end{align}
The $T$-invariant ergodic measures of full dimension for $K$ are
the ergodic equilibrium states for $-\alpha \Phi\circ\pi.$
\end{theorem}
\begin{proof}
Using the proof of Corollary 2.7 \cite{Yayama} and Corollary 3.3
\cite{KP}, we obtain
\begin{equation*}
\begin{split}
\dim_{H}{K} &= \sup \{\dim _{H}\mu: \mu (K)=1, \mu \text { is } T
-\text{invariant and ergodic.}\} (\text{by Theorem 3 \cite{GP}})
\\&=\frac{1}{\log l} \sup_{\mu \in M(X,
\sigma_X)}\{h_{\mu}(\sigma_X)+(\log_{m}l-1)h_{\pi
\mu}(\sigma_Y)\}.
\end{split}
\end{equation*}
Setting $\alpha=\log_{m}l-1$ and applying Proposition
\ref{sftwentropy} establishes equalities (\ref{HDformula1}) and
(\ref{HDformula2}). For equality (\ref{HDformula3}), we use the
variational principle (Theorem \ref{CP}). For the last part, we apply Theorem \ref{thm2} and
Proposition \ref{sftwentropy}.
\end{proof}
\begin{lemma}\label{suba}
Let $(X,\sigma_X),(Y,\sigma_Y)$ be subshifts and $\pi:(X,\sigma_X)
\rightarrow (Y,\sigma_Y)$ a factor map. Let $\alpha>0$. For all $n\in \N$, define $\phi^{'}_n:Y\rightarrow \R$ by $\phi^{'}_n(y)=\log \vert D_n(y)\vert^{1/(\alpha+1)}.$ Then for $n,m\in \N, y\in Y$, $\phi^{'}_{n+m}(y)\leq \phi^{'}_{n}(y)+\phi^{'}_m({\sigma^{m}_{Y}}(y)).$
\end{lemma}
\begin{proof}
It is enough to show that for $y\in Y, m,n\in \N,$
\begin{equation}\label{dneq}
\vert D_{n+m}(y)\vert \leq \vert D_n(y)\vert \vert D_m(\sigma_Y^{n}(y))\vert.
\end{equation}
\noindent Let $z=z_1\cdots z_n z_{n+1}\cdots z_{n+m}\cdots \in
D_{n+m}(y).$ Then the number of points in $D_{n+m}(y)$ starting
with $z_1\cdots z_n$ is less than or equal to $\vert
D_m(\sigma_Y^n(y))\vert$. Since the number of all possible
distinct $z_1\cdots z_n$ of the points in $D_{n+m}(y)$ is equal to
$\vert D_n(y)\vert,$ we obtain inequality (\ref{dneq}).
\end{proof}
\begin{lemma}\label{Dn}
Let $(X,\sigma_X),(Y,\sigma_Y)$ be subshifts and $\pi:(X,\sigma_X)
\rightarrow (Y,\sigma_Y)$ a factor map.  For all $n\in \N,$ define
$\phi_n^{'}$ as in Lemma \ref{suba} and let
$\Phi^{'}=\{\phi_n^{'}\}^{\infty}_{n=1}.$ Let $S$ be the number of symbols in $X.$ For $0<\epsilon <1/2$,
let $k \geq 2$ such that $1/2^{k}\leq \epsilon < 1/2^{k-1}.$ If $\phi_n^{'}$ is continuous for all $n\in \N,$ then
$$\sum_{y_1 \cdots y_{n+k-2} }\big(\frac{\vert \pi^{-1}[y_1 \cdots y_{n+k-2}]\vert}{S}\big )^{\frac{1}{\alpha+1}} \leq P_{n}(\sigma_Y, \Phi^{'}, \epsilon)\leq \sum_{y_1\cdots y_{n+k-2}}\vert \pi^{-1}[y_1\cdots y_{n+k-2}]\vert ^{\frac{1}{\alpha+1}},$$
where the summations are taken over all allowable words $y_1 \cdots y_{n+k-2}$ of length $(n+k-2)$ in $Y$.
\end{lemma}
\begin{proof}
Using Lemma \ref{suba}, $\Phi^{'}$ is a subadditive potential on $(Y, \sigma_Y)$ if $\phi_n^{'}$ is continuous for all $n\in \N.$
Let $\epsilon >0$ be fixed. Take an $(n,\epsilon)$ separated subset $A$ of $Y$. If $x,y \in A,$ then there exists $1\leq i \leq
n+k-2$ such that $x_i \neq y_i.$ Thus
$$\sum_{y\in A}\vert D_{n+k-2}(y)\vert^{1/(\alpha +1)}\leq \sum_{y_1 \cdots y_{n+k-2}\in B_{n+k-2}(Y)}\vert \pi^{-1}[y_1 \cdots y_{n+k-2}]\vert ^{1/(\alpha +1)}.$$
Let $[i_1\cdots i_{n+k-2}]$ be a cylinder in $Y$ and let $y\in  [i_1 \cdots i_{n+k-2}].$
Then the number of
possible choices of symbols in the $(n+k-2)$ th position of a point in
$E_{n+k-2}(y)\leq S.$  Therefore, there
exists a symbol $a$ in the $(n+k-2)$ th position of a point in $ E_{n+k-2}(y)$ such that
\begin{equation*}
\begin{split}
\frac {\vert \pi^{-1}[i_1 \cdots i_{n+k-2} ]\vert}{S}& \leq
\textnormal{the total number of }a \textnormal{ that appears in
the } (n+k-2) \textnormal{ th}\\& \textnormal{ position of points
in }E_{n+k-2}(y) .
\end{split}
\end{equation*}
Take $x \in E_{n+k-2}(y)$ such that
$x_{n+k-2}=a.$ Let $\pi(x)=z.$ Then
$$\frac {\vert \pi^{-1}[i_1  \cdots i_{n+k-2}]\vert}{S} \leq \vert D_{n+k-2}(z)\vert.$$
Now let $E$ be the set obtained by taking one such $z$ from each distinct cylinder $[i_1 \cdots i_{n+k-2}]$ in $Y$. Then $E$ is an $(n,\epsilon)$ separated subset of $Y$ and
\begin{equation*}
\sum_{y_1 \cdots y_{n+k-2} \in B_{n+k-2}(Y)}\big(\frac{\vert \pi^{-1}[y_1 \cdots y_{n+k-2}]\vert}{S}\big)^{1/(\alpha+1)} \leq \sum_{y\in E}\vert D_{n+k-2}(y)\vert^{{1}/(\alpha+1)}.
\end{equation*}
Now Lemma \ref{Dn} is proved.
\end{proof}

Now we consider a sequence of real-valued functions $\Phi=\{\phi_n\}_{n=1}^{\infty}$ on $(X, \sigma_X)$ such that $\Phi$ satisfies the subadditivity condition but $\phi_n$ is merely Borel measurable on $X$. Then $-\Phi$ is a sequence of Borel measurable functions on $(X, \sigma_X)$ such that $-\Phi$ satisfies the superadditivity condition. We define the equilibrium states for $\Phi$ and $-\Phi$ in the same manner as we did for subadditive potentials and superadditve potentials (see pages \pageref{defesp} and \pageref{defespp}).
\begin{theorem}\label{application1}
Let $(X,\sigma_X),(Y,\sigma_Y)$ be subshifts and $\pi:(X,\sigma_X)
\rightarrow (Y,\sigma_Y)$ a factor map.  For $l,m\in \N, l>m\geq 2
,$ set $\alpha=\log _{m}l-1.$ For all $n\in\N,$ define $\phi_n$
and $\Phi$ as in Theorem \ref{HDDFORSFT} and $\phi_n^{'}$ and
$\Phi^{'}$ as in Lemma \ref{Dn}.  Then the Hausdorff dimension of
a subshift-NC carpet $K$ corresponding to $(X,Y,\pi)$ is given by
\begin{align}
\dim_{H}K &=\frac{1}{\log m}\sup_{\nu \in M(Y,\sigma_Y)} \{h_{\nu}(\sigma_Y)+\lim_{n\rightarrow \infty}\frac{1}{n}\int \phi_n^{'} d\nu\} \label{HDformula1.1}\\=& \frac{1}{\log m}\sup_{\mu \in M(X,\sigma_X)} \{h_{\mu}(\sigma_X)-\lim_{n\rightarrow \infty}\frac{1}{n}\int \alpha \phi_n^{'}\circ\pi d\mu\}\label{HDformula1.2}.
\end{align}
The $T$-invariant ergodic measures of full dimension for $K$ are
the ergodic equilibrium states for $-\alpha \Phi^{'}\circ \pi$. In
addition, if $\phi_n^{'}$ is continuous for all $n\in \N$, then
\begin{enumerate}
\item $\dim_{H}K =(1/\log m) \limsup_{n\rightarrow \infty} (1/n) \log
(\sum_{y_1 \cdots y_n \in B_n(Y)}  \vert \pi^{-1}[y_1\cdots
y_n]\vert^{1/(\alpha+1)})$.\label{subshiftwc}\\
\item $M_{\Phi'}(Y,\sigma_Y)\subseteq M_{\Phi}(Y,\sigma_Y)$.\label{compare}
In particular, if $ M_{\Phi}(Y,\sigma_Y) $ consists of one point, then $M_{\Phi'}(Y,\sigma_Y)= M_{\Phi}(Y,\sigma_Y)$.\\
\item  If $M_{\Phi}(Y,\sigma_Y)$ consists of one point, then  the $T$-invariant ergodic measures of full dimension for $K$ are the ergodic equilibrium states for $ -\alpha \Phi\circ \pi$.\label{nicesubshift}
\end{enumerate}
\end{theorem}
\begin{proof}
We fist notice that the map $y\rightarrow \log \vert D_n(y)\vert$
is Borel measurable for each $n\geq 1,$ by the proof of Lemma 3.3
(with $f=0$) in \cite{LW}. Equality (\ref{HDformula1.1}) is clear
from the proof of (\ref{HDformula1}) in Theorem \ref{HDDFORSFT}.
Equality (\ref{HDformula1.2}) follows by combining Theorem
\ref{wentropy}, Lemma \ref{suba} and Subadditive Ergodic Thereom.
Applying Theorem \ref{wentropy}, the $T$-invariant ergodic
measures of full dimension for $K$ are the ergodic equilibrium
states for $-\alpha \Phi^{'}\circ \pi$. By Lemma \ref{Dn},
\begin{equation*}
P(\sigma_Y, \Phi^{'}, \epsilon)=\limsup_{n\rightarrow
\infty}\frac{1}{n}\log P_n(\sigma_Y, \Phi^{'}, \epsilon)
=\limsup_{n \rightarrow \infty}\frac{1}{n}\log \big (\sum_{y_1
\cdots y_n \in B_n(Y)}\vert \pi^{-1}[y_1\cdots y_n]\vert
^{1/(\alpha+1)}\big).
\end{equation*}
Using the variational principle, the
above implies (\ref{subshiftwc}). For (\ref{compare}), we notice
by (\ref{subshiftwc}) that if $\phi'_{n}$ is continuous for all
$n\in \N$, then
$$\sup_{\nu \in M(Y,\sigma_Y)}\{h_{\nu}(\sigma_Y)+\lim_{n\rightarrow \infty}\frac{1}{n}\int \phi_n^{'} d\nu\}= \sup_{\nu \in M(Y,\sigma_Y)}\{h_{\nu}(\sigma_Y)+\lim_{n\rightarrow \infty}\frac{1}{n}\int \phi_n  d\nu\}.$$
Using $\phi'_n\leq \phi_n$ for all $n\in\N$, we obtain $M_{\Phi'}(Y,\sigma_Y)\subseteq M_{\Phi}(Y,\sigma_Y).$ Since $M_{\Phi'}(Y,\sigma_Y) \neq \emptyset,$ if
 $M_{\Phi}(Y,\sigma_Y)$ consists of one point, then $M_{\Phi'}(Y,\sigma_Y)= M_{\Phi}(Y,\sigma_Y).$ (\ref{nicesubshift}) follows by
 combining Theorem \ref{thm2} and Theorem \ref{wentropy}.
\end{proof}
Next we want to know whether or not there is a unique invariant
ergodic measure of full dimension and to study the properties of
the unique measures. In \cite{Yayama}, uniqueness and the
properties of the unique measure were studied for some SFT-NC
carpets for which continuous saturated compensation functions
exist. In order to generalize these results, we follow the proofs
of Lemmas 1 and 2 and Theorem 5 in \cite{B2006}.

Throughout the rest of this section, we assume that $(X,
\sigma_X)$ and  $(Y, \sigma_X)$ are topologically mixing shifts of
finite type. Let $A$ be the transition matrix for $X$ and $M$ the
smallest integer such that $A^{M}>0.$ Let $\alpha>0.$  Define
$K=\sum_{i_1 \cdots i_M \in B_{M}(Y)}\vert \pi^{-1}[i_1 \cdots
i_M]\vert^{1/(\alpha+1)}$. For $n\in \N,$ define $S_n=
\sum_{i_1\cdots i_n \in B_{n}(Y)} \vert\pi^{-1}[i_1 \cdots
i_n]\vert^{1/(\alpha+1)}.$   Define ${\phi_n}:Y\rightarrow \R$ by
${\phi_n}(y)= \log \vert \pi^{-1}[y_1 \cdots y_n]
\vert^{1/(\alpha+1)}$ for all $n\in \N$ and
$\Phi=\{\phi_n\}^{\infty}_{n=1}.$ Then $\Phi$ is a subadditive
potential on $(Y, \sigma_Y)$ with bounded variation. $\Phi$ is
almost additive if and only if  for any allowable word $i_1 \cdot
\cdot \cdot i_n j_1 \cdots j_l$ of length $(n+l)$ in $Y$, $n,l\in
\N$,  there exist $K_1, K_2>0$ such that
\begin{align} \label{aadditiveinourcase}
K_1\leq \frac{\vert \pi^{-1}[i_{1}\cdots i_{n}j_{1}\cdots
j_{n}\vert]}{\vert \pi^{-1}[i_1 \cdots i_n]\vert
\vert\pi^{-1}[j_1\cdots j_l]\vert} \leq K_2,
\end{align}
We notice that $\Phi$ is not always almost additive. Therefore the
hypothesis of Theorem \ref{BM} is not satisfied. Nevertheless, we
will show that there is a unique equilibrium state for $\Phi$
which is Gibbs and mixing. We use the approach in \cite{B2006} to
finding uniqueness of the equilibrium state for an almost
subadditive potential. We prove that Lemmas 1 and 2, and therefore
Theorem 5 \cite{B2006} still hold for our subadditive potential
$\Phi$ on $(Y,\sigma_Y)$. For all $n\in \N$, let $A_n$ be a set
consisting exactly one point from each cylinder of length $n$ in
$Y$. Define the Borel probability measure $\nu_n$ on $Y$
concentrated on $A_n$ by
$$\nu_n=\frac{\sum_{y\in A_n}e^{\phi_n(y)}\delta_{y}}{\sum_{y\in A_n}e^{\phi_{n}(y)}},$$
where $\delta_{y}$ is the Dirac measure at $y$. Then, for each
cylinder $[i_1\cdots i_n]$ of length $n$ in $Y$,
$$\nu_n([i_1\cdots i_n])=\frac{\vert \pi^{-1}[i_1 \cdots i_n]\vert^{1/(\alpha+1)}}{\sum_{j_1\cdots j_n\in B_n(Y)}  \vert\pi^{-1}[j_1 \cdots j_n]\vert^{1/(\alpha+1)}}.$$
Since $\nu_n$ is a Borel probability measure on $Y$ for all $n\in
\N$, there exists a subsequence $\{\nu_{n_k}\}_{k=1}^{\infty}$
that converges to a Borel probability measure $\nu$ on $Y$ in the
weak* topology. In the following lemmas, let $K, S_n, \Phi$ and $
\nu_n$ be defined as above.
\begin{lemma}\label{key2}
For all $n\in \N$, there exist $K_1, K_2>0$ such that $K_1\leq
{e^{nP(\sigma_Y, \Phi)}}/{S_n}\leq K_2.$
\end{lemma}
\begin{proof}
Notice that $P(\sigma_Y,\Phi)=\limsup_{n\rightarrow \infty}
(\log S_n)/{n}$ by the subadditive topological pressure of $\Phi.$
We will first find $K_2>0$. Since
\begin{equation*}
\begin{split}
&\sum_{i_1\cdots i_n j_1 \cdots j_l \in B_{n+l}(Y)}
\vert\pi^{-1}[i_1 \cdots i_n j_1 \cdots j_l
]\vert^{1/(\alpha+1)}\\&\leq \sum_{i_1\cdots i_n \in B_n(Y)}
\vert\pi^{-1}[i_1 \cdots i_n]\vert^{1/(\alpha+1)} \sum_{j_1\cdots
j_l \in B_{l}(Y)}  \vert\pi^{-1}[j_1 \cdots
j_l]\vert^{1/(\alpha+1)},\label{trivial}
\end{split}
\end{equation*}
$\{\log S_n\}_{n=1}^{\infty}$ is subadditive. Therefore,
$$P(\sigma_Y, \Phi)=\lim_{n\rightarrow \infty}\frac{\log S_n}{n}\leq \frac{\log S_n}{n}
\textnormal { for all  } n\geq 1,$$ and so ${e^{nP(\sigma_Y,
\Phi)}}/{S_n}\leq 1.$ Set $K_2=1.$ Now let $l>M$. To find $K_1>0$,
we first show that $S_{l+n}\geq S_{n} S_{l-M}.$ Given any two
allowable words in $Y$, $i_{1} \cdots i_{n}$ of length $n$,
$j_{1}\cdots j_{l-M}$ of length $(l-M),$ there exists an allowable
word $a_1 \cdots a_M$ of length $M$ in $Y$ such that $i_1 \cdots
i_n a_1 \cdots a_M j_1 \cdots j_{l-M}$ is an allowable word of
length $(l+n)$ in $Y.$ For fixed allowable words $i_1\cdots i_n$
of length $n$ and $j_1 \cdots j_{l-M}$ of length $(l-M),$
$$\sum_{i_1\cdots i_n a_1 \cdots a_M j_1 \cdots j_{l-M}\in B_{n+l}(Y)}\vert \pi^{-1}[i_1\cdots i_n a_1\cdots a_M j_1 \cdots j_{l-M}]\vert \geq \vert \pi^{-1}[i_1\cdots i_n]\vert  \vert \pi^{-1}[j_1\cdots j_{l-M}]\vert.$$
Therefore,
\begin{equation*}
\begin{split}
& \sum_{i_1\cdots i_n a_1\cdots a_M j_1 \cdots j_{l-M}\in
B_{n+l}(Y)}\vert \pi^{-1}[i_1\cdots i_n a_1 \cdots a_M j_1 \cdots
j_{l-M}]\vert^{1/(\alpha+1)}\\ & \geq \big(\sum_{i_1\cdots i_n
a_1\cdots a_M j_1 \cdots j_{l-M}\in B_{n+l}(Y)}\vert
\pi^{-1}[i_1\cdots i_n a_1\cdots a_M j_1 \cdots
j_{l-M}]\vert \big)^{1/(\alpha+1)}\\& \geq \vert \pi^{-1}[i_1\cdots
i_n]\vert ^{1/(\alpha+1)} \vert \pi^{-1}[j_1\cdots
j_{l-M}]\vert^{1/(\alpha+1)}.
\end{split}
\end{equation*}
Summing over all allowable words $i_1\cdots i_n$ of length $n$ in
$Y$ and $j_1\cdots j_{l-M}$ of length $(l-M)$ in $Y$, we obtain
\begin{align}\label{key2-1}
S_{l+n}\geq S_n S_{l-M}.
\end{align}
Now we show $S_l\leq K S_{l-M}.$ For any allowable word
$i_{M+1}\cdots i_l$ in $Y,$ there exists $a_1 \cdots a_M$ such
that $ a_1 \cdots a_M i_{M+1}\cdots i_l$ is an allowable word of
length $l$ in $Y.$ Let $i_{M+1} \cdots i_{l}$ be fixed. Then
$$ \sum_{i_1 \cdots i_M i_{M+1} \cdots i_{l}\in B_{l}(Y)} \vert \pi^{-1}[i_1 \cdots  i_M i_{M+1} \cdots i_{l}]\vert^{1/(\alpha+1)}  \leq
K \vert \pi^{-1}[i_{M+1}\cdots i_{l}]\vert^{1/(\alpha+1)}.$$
Summing all allowable word $i_{M+1}\cdots i_{l}$ of length $(l-M)$
in $Y$, we get
\begin{align}\label{key2-2}
S_l\leq KS_{l-M}.
\end{align}
By inequalities (\ref{key2-1}) and (\ref{key2-2}), $S_{l+n}\geq
S_{l}S_{n}/K$ for $l>M, n\geq 1$. For $l+n\leq 2M,$ we can find $K'
\in \N$ such that $S_{l+n}\geq
S_{l}S_{n}/K',$ because there are only finitely many choices of $(l,n)$ such that $l+n\leq 2M.$ Set $\widetilde K=\max\{K, K'\}$. Then $\{\log (S_n/\widetilde{K})\}_{n=1}^{\infty}$ is
superadditive. Thus
\begin{equation*}
\lim_{n\rightarrow \infty}\frac{1}{n}\log{S_n} =\lim_{n\rightarrow
\infty}\frac{1}{n}\log \frac{S_n}{\widetilde{K}} \geq
\frac{1}{n}\log \frac{S_n}{\widetilde{K}} \textnormal{ for all }
n\geq 1.
\end{equation*}
Using $P(\sigma_Y, \Phi)=\lim_{n \rightarrow \infty} (\log S_n)
/n,$ we get  ${1}/{\widetilde{K}}\leq {e^{nP(\sigma_Y,
\Phi)}}/{S_n}.$ Set $K_1= {1}/{\widetilde{K}}.$
\end{proof}
\begin{lemma}\label{gibbsfornul}
For all $l, n \in \N, l > n+M,$ and cylinders $[i_1\cdots i_n]$ in
$Y,$ there exist $C_1, C_2>0$ such that
\begin{equation*}
C_1 \leq \frac{{\nu_{l}}([i_1 \cdots i_n])}{e^{-nP(\sigma_Y,
\Phi)}\vert \pi^{-1}[i_1 \cdots i_n]\vert^{1/(\alpha+1)}}\leq C_2.
\end{equation*}
\end{lemma}
\begin{proof}
Let $[i_1\cdots i_n]$ be a fixed cylinder of length $n$ in $Y.$ By the definition of ${\nu}_l,$ for $n<l,$
\begin{align}\label{defofnul}
\nu_l([i_1 \cdots i_n])=\frac {\sum_{i_1 \cdots i_n j_1 \cdots
j_{l-n} \in B_{l}(Y)}\vert \pi^{-1}[i_1 \cdots i_n j_1 \cdots
j_{l-n}]\vert ^{1/(\alpha +1)}}{\sum_{i_1 \cdots i_l \in
B_{l}(Y)}\vert\pi^{-1}[i_1 \cdots i_l]\vert ^{1/(\alpha +1)}}.
\end{align}

We first find an upperbound $C_2>0.$ For $l>n+M,$ using the
property of topologically mixing,
\begin{align}
& \sum _{i_1\cdots i_n j_1\cdots j_{l-n}\in
B_{l}(Y)}\vert\pi^{-1}[i_1\cdots i_n j_{1}\cdots j_{l-n}]
\vert^{{1}/(\alpha+1)}
\\ & \leq K\vert \pi^{-1} [i_1\cdots i_n]\vert ^{1/(\alpha+1)}\sum_{j_{M+1}\cdots j_{l-n}\in B_{l-n-M}(Y)}\vert \pi^{-1}[j_{M+1}\cdots j_{l-n}]\vert^{1/(\alpha+1)}\label{e1fortopmix}.
\end{align}
Therefore, using (\ref{defofnul}), (\ref{e1fortopmix}) and Lemma
\ref{key2} (with $K_1=1/\widetilde{K}$ and $K_2=1$),
\begin{equation*}
\begin{split}
\frac{{\nu_{l}}([i_1 \cdots i_n])}{e^{-nP(\sigma_Y,\Phi)}\vert
\pi^{-1}[i_1 \cdots i_n]\vert^{1/(\alpha+1)}} &\leq
\frac{KS_{l-n-M}}{S_l}e^{nP(\sigma_Y, \Phi)} \leq
K\cdot \widetilde{K}e^{-MP(\sigma_Y, \Phi)} \textnormal { (by Lemma \ref{key2})}.
\end{split}
\end{equation*}
Next we will find a lower bound $C_1>0$. For two fixed allowable
words $i_1 \cdots i_n$ of length $n$ and $j_{M+1}\cdots j_{l-n}$
of length $(l-n-M)$ in $Y,$ there exists $a_1 \cdots a_M$ such
that $i_1 \cdots i_n a_1 \cdots a_M j_{M+1} \cdots j_{l-n}$ is an
allowable word of length $l$ in $Y.$ Then
\begin{equation*}
\begin{split}
&\sum_{i_1\cdots i_n a_1\cdots a_M j_{M+1}\cdots j_{l-n}\in
B_{l}(Y)}\vert \pi^{-1}[i_1 \cdots i_{n} a_1 \cdots a_M j_{M+1}
\cdots j_{l-n}]\vert^{1/(\alpha+1)}\\& \geq \vert
\pi^{-1}[i_1\cdots i_n]\vert ^{1/(\alpha+1)}\vert
\pi^{-1}[j_{M+1}\cdots j_{l-n}]\vert^{{1}/(\alpha+1)}.
\end{split}
\end{equation*}
Summing over all allowable words $j_{M+1}\cdots j_{l-n}$ of length
$(l-n-M)$ in $Y,$ we get
\begin{align}
& \sum_{i_1 \cdots i_n j_1\cdots j_{l-n}\in B_{l}(Y) }\vert \pi^{-1}[i_1\cdots i_n j_1 \cdots j_{l-n}]\vert^{1/(\alpha+1)}\\
 & \label{eq2fortopmix} \geq
\vert \pi^{-1}[i_1\cdots i_n]\vert ^{1/(\alpha+1)}
\sum_{j_{M+1}\cdots j_{l-n}\in B_{l-n-M}(Y)}\vert
\pi^{-1}[j_{M+1}\cdots j_{l-n}]\vert^{1/(\alpha+1)}.
\end{align}
Therefore, using inequality (\ref{eq2fortopmix}) and Lemma
\ref{key2},
\begin{equation*}
\frac{{\nu_{l}}([i_1 \cdots i_n])}{e^{-nP(\sigma_Y, \Phi)}\vert
\pi^{-1}[i_1 \cdots i_n]\vert^{{1}/(\alpha+1)}} \geq
\frac{S_{l-n-M}}{S_l} e^{nP(\sigma_Y, \Phi)}\geq
\frac{S_{l-n-M}S_n}{\widetilde{K}S_l}\geq \frac{e^{-MP(\sigma_Y, \Phi)}}{\widetilde{K}^2}.
\end{equation*}
\end{proof}

\noindent By Lemma \ref{gibbsfornul}, if a subsequence $\{\nu_{n_k}\}^{\infty}_{k=1}$ of $\{\nu_{n}\}_{n=1}^{\infty}$
converges to a Borel probability measure $\nu$ on $Y$ in the weak* topology, then
\begin{align}
C_1 \leq \frac{{\nu}([i_1 \cdots i_n])}{e^{-nP(\sigma_Y,
\Phi)}\vert \pi^{-1}[i_1 \cdots i_n]\vert^{1/(\alpha+1)}}\leq C_2 \textnormal{ for all }n\in\N.
\label{vgibbs}
\end{align}
\begin{lemma}\label{equistateY}
Let $\nu$ be the limit point of a convergent subsequence
$\{\nu_{n_k}\}^{\infty}_{k=1}$ of  $\{\nu_{n}\}^{\infty}_{n=1}$.
Let $\mu_n=\frac{1}{n}\sum_{i=0}^{n-1}\sigma_Y^{i}\nu.$ Then any
weak limit point $\mu$ of $\{\mu_{n}\}_{n=1}^{\infty}$ is a
$\sigma_Y$-invariant Gibbs measure for $\Phi.$
\end{lemma}
\begin{proof}
We follow the arguments in the first part of the proof of Theorem
1 \cite{B2006}. Suppose $\{\mu_{n_k}\}_{k=1}^{\infty}$ converges
to $\mu$ in the weak* topology. Then $\mu$ is a
$\sigma_Y$-invariant Borel probability measure on $Y$ (Theorem
6.9. in \cite{Wtext}). We want to show that $\mu$ is Gibbs.  Let
$i_1\cdots i_n$ be a fixed allowable word of length $n$ in $Y.$
For each $l,n\in \N,l>M$,
\begin{align}
(\sigma_Y^{l}\nu)([i_1\cdots i_n])&=\sum_{j_1 \cdots j_l
i_{1}\cdots i_{n} \in B_{l+n}(Y)} \nu([j_1 \cdots j_l i_1 \cdots
i_n])\\& \geq \sum_{j_1 \cdots j_l i_{1}\cdots i_{n} \in
B_{l+n}(Y)} C_1e^{-(l+n)P(\sigma_Y, \Phi)}\vert \pi^{-1}[j_1\cdots
j_li_1\cdots i_n]\vert^{\frac{1}{\alpha+1}}  \\&\geq C_1 \vert \pi^{-1}[i_1\cdots
i_n]\vert^{\frac{1}{\alpha+1}}e^{-(l+n)P(\sigma_Y, \Phi)}\sum_{j_1
\cdots j_{l-M} \in B_{l-M}(Y)}\vert \pi^{-1}[j_1 \cdots
j_{l-M}]\vert^{\frac{1}{\alpha+1}} \label{eq3fortopomix}\\& \geq
e^{-MP(\sigma_Y, \Phi)}\frac{C_1}{C_2}\nu([i_1\cdots
i_n])(\textnormal{by Lemma \ref{key2} and (\ref{vgibbs})}) .
\end{align}
For inequality (\ref{eq3fortopomix}), we use the following
inequality (\ref{topomixf}) which  is easy to show.

For a fixed allowable word $i_1 \cdots i_n $ of length $n$,
$l>M,$
\begin{equation}\label{topomixf}
\begin{split}
\sum_{j_1 \cdots j_l i_1 \cdots i_n\in B_{l+n}(Y)}\vert
\pi^{-1}[& j_1 \cdots j_l i_1 \cdots i_n]\vert^{1/(\alpha+1)} \\
& \geq
\vert \pi^{-1}[i_1 \cdots i_n]\vert^{1/(\alpha+1)}\sum_{j_1\cdots
j_{l-M}\in B_{l-M}(Y)}\vert \pi^{-1}[j_1 \cdots j_{l-M}]\vert^{
1/(\alpha+1)}.
\end{split}
\end{equation}

For a fixed allowable word $i_1\cdots i_n$ of length $n$ in $Y,$
\begin{equation*}
\begin{split}
(\sigma_Y^{l}\nu)([i_1\cdots i_n]) &\leq \sum_{j_1 \cdots j_l
i_1\cdots i_n \in B_{l+n}(Y)}C_2 e^{-(n+l)P(\sigma_Y, \Phi)}\vert
\pi^{-1}[j_1 \cdots j_l i_1 \cdots i_n]\vert^{1/(\alpha +1)}\\&
\leq C_2 \frac{\sum_{j_1\cdots j_l \in B_{l}(Y)} \vert
\pi^{-1}[j_1 \cdots j_l]\vert ^{1/(\alpha
+1)}}{e^{(l+n)P(\sigma_Y, \Phi)}}\vert \pi^{-1}[i_1\cdots
i_n]\vert^{1/(\alpha+1)}\\ &\leq \frac{\widetilde{K}C_2}{C_1}\nu[i_1\cdots
i_n] \textnormal{(by Lemma \ref{key2} and
(\ref{vgibbs}))}.
\end{split}
\end{equation*}
Therefore, for all $m> M+1$, using (\ref{vgibbs}), we obtain
$$\frac{m-M-1}{m}\big(\frac{C_1^2}{C_2e^{(M+n)P(\sigma_Y, \Phi)}} \vert \pi^{-1}[i_1\cdots i_n]\vert^{1/(\alpha+1)}\big)\leq \frac{1}{m}\sum_{l=0}^{m-1}(\sigma_Y^{l}\nu)([i_1\cdots i_n])$$ and
$$\frac{1}{m}\sum_{l=0}^{m-1}(\sigma_Y^{l}\nu)([i_1\cdots i_n]) \leq  \frac{m-M-1}{m}\big(\frac{\widetilde{K}{C_2}^2}{C_1 e^{nP(\sigma_Y, \Phi)}}\vert \pi^{-1}[i_1\cdots i_n]\vert^{1/(\alpha+1)}\big)+\frac{M+1}{m}.$$
Since $\{\mu_{n_k}\}_{k=1}^{\infty}$ converges to $\mu$ in the weak*
topology, replacing $m$ by $n_k$ and letting $k\rightarrow \infty$,
$$\frac{C_1^2}{C_2e^{MP(\sigma_Y, \Phi)}}\leq \frac{\mu[i_1\cdots i_n]}{e^{-nP(\sigma_Y, \Phi)}\vert \pi^{-1}[i_1\cdots i_n]\vert^{1/(\alpha+1)}} \leq \frac{\widetilde{K}{C_2}^2}{C_1}.$$
Therefore $\mu$ is Gibbs for $\Phi$.
\end{proof}
By Lemma \ref{equistateY}, it is easy to see that $\mu$ is an equilibrium state for
$\Phi$ by using similar arguments as in the proof of Lemma 17 \cite{m2006}.
\begin{lemma}\label{topmix5}
For fixed allowable words $i_1\cdots i_n, j_1\cdots j_l$ in $Y,
k>n+2M,$
\begin{equation*}
\begin{split}
&\sum _{i_1\cdots i_n a_1 \cdots a_{k-n}j_1\cdots j_l \in
B_{k+l}(Y)}\vert \pi^{-1}[i_1\cdots i_n a_1 \cdots a_{k-n}j_1
\cdots j_l]\vert^{1/(\alpha+1)}
\\& \geq (\vert \pi^{-1}[i_1 \cdots i_{n}]\vert \vert \pi^{-1}[j_1 \cdots j_{l}]\vert)^{1/(\alpha+1)}\sum _{b_1\cdots b_{k-n-2M}\in B_{k-n-2M}(Y)}\vert \pi^{-1}[b_1 \cdots b_{k-n-2M}]\vert^{1/(\alpha+1)}.
\end{split}
\end{equation*}
\end{lemma}
\begin{proof}
Let $a_{M+1} \cdots a_{k-n-M}$ be an allowable word of length
$(k-n-2M)$ in $Y.$ Call it $c$. Then there exist $\overline {a_1 \cdots a_M},
\overline{a_{k-n-M+1} \cdots a_{k-n}}$ such that $i_1 \cdots i_n
\overline{a_1 \cdots a_M} c
\overline{a_{k-n-M+1} \cdots a_{k-n} }j_1 \cdots j_l$ is allowable
in $Y$. Denote $ \overline{a_{1} \cdots a_{M}}$ by $u$ and $
\overline{a_{k-n-M+1}\cdots a_{k-n}}$ by $v.$ Fix $c$ and $v$. Then
\begin{equation*}
\sum _{i_1 \cdots i_n  ucv  j_1 \cdots j_l
\in B_{k+l}(Y)}\vert \pi^{-1}[i_1 \cdots i_n u
cv
j_1 \cdots j_l] \vert \geq \vert \pi^{-1}[i_1 \cdots
i_{n}]\vert \vert \pi^{-1}[cv j_1\cdots j_l]\vert.
\end{equation*}
Therefore,
\begin{equation*}
\sum _{i_1\cdots i_n ucv j_1\cdots j_l \in B_{k+l}(Y) }\vert \pi^{-1}[i_1\cdots i_n ucv j_1\cdots j_l]\vert ^{1/(\alpha +1)}
\geq (\vert \pi^{-1}[i_1 \cdots i_{n}] \vert \vert\pi^{-1}[cv j_1\cdots j_l]\vert)^{1/(\alpha +1)}.
\end{equation*}
Now for fixed $i_1 \cdots i_n, j_1 \cdots j_l$ and  $c$, summing over all allowable words $u,
v$ in $Y$ such that $i_1\cdots
i_n ucv j_1\cdots j_l$ is allowable,
\begin{equation*}
\begin{split}
&\sum_{i_1\cdots i_n ucv j_1\cdots j_l \in B_{k+l}(Y)} \vert \pi^{-1}[i_1\cdots i_n ucv j_1 \cdots j_l]\vert ^{1/(\alpha +1)}
\\& \geq \vert\pi^{-1}[i_1 \cdots i_n]\vert^{1/(\alpha+1)}
\sum_{cvj_1 \cdots j_l\in B_{k+l-n-M}(Y)}\vert
\pi^{-1}[cvj_{1}\cdots j_{l}]\vert ^{1/(\alpha+1)} \\& \geq \vert
\pi^{-1}[i_1\cdots i_n]\vert^{1/(\alpha+1)} \vert
\pi^{-1}[c]\vert^{1/(\alpha+1)}\vert
\pi^{-1}[j_1 \cdots j_l] \vert^{1/(\alpha+1)}.
\end{split}
\end{equation*}
Finally summing over all allowable words $u,c,v$  in
$Y$ such that $i_1\cdots i_n ucv j_1\cdots j_l$ is allowable,
\begin{equation*}
\begin{split}
&\sum_{i_1\cdots i_n ucv j_1\cdots j_l \in B_{k+l}(Y)} \vert \pi^{-1}[i_1\cdots i_n ucv j_1 \cdots j_l]\vert ^{1/(\alpha +1)} \\
& \geq (\vert \pi^{-1}[i_1 \cdots i_n]\vert \vert
\pi^{-1}[j_1 \cdots j_l]\vert)^{1/(\alpha+1)} \sum_{a_{M+1}\cdots
a_{k-n-M}\in B_{k-n-2M}(Y)} \vert \pi^{-1}[a_{M+1}\cdots
a_{k-n-M}] \vert^{1/(\alpha+1)}.
\end{split}
\end{equation*}
\end{proof}
\begin{lemma}\label{ergodic}
Let $\nu$ be a Gibbs measure for $\Phi.$ Then any Gibbs measure
for $\Phi$ is ergodic.
\end{lemma}
\begin{proof}
We use the same arguments as in the proof of Lemma 2 in
\cite{B2006}. We show that for any two cylinders $[i_1\cdots
i_n],$ $[j_1 \cdots j_l],$ for each $k>n+2M,$ there exists $C>0$
such that $\nu([i_1 \cdots i_n] \cap \sigma_{Y}^{-k}([j_1 \cdots
j_l]))\geq C\nu([i_1\cdots i_n])\nu([j_1\cdots j_l]).$ This
implies that $\liminf_{k\rightarrow\infty}\nu(A \cap
\sigma_{Y}^{-k}(B))\geq C\nu(A)\nu(B)$ for any Borel measurable
subsets $A, B$ of $Y$. Suppose $\nu$ is a Gibbs measure for $\Phi$
as in (\ref{vgibbs}). We denote $i_1\cdots i_n$ by $u_1$ and $j_1
\cdots j_l$ by $u_2$.

\noindent Using Lemmas \ref{key2} and \ref{topmix5},
\begin{equation*}
\begin{split}
&\nu([u_1] \cap \sigma_{Y}^{-k}([u_2]))= \sum _{u_1 a_1 \cdots a_{k-n}u_2 \in B_{k+l}(Y)}\nu([u_1 a_1 \cdots a_{k-n} u_2])\\
& \geq \sum_{u_1 a_1 \cdots a_{k-n}u_2\in B_{k+l}(Y)}C_1 \vert \pi^{-1}[u_1 a_1 \cdots a_{k-n}u_2]\vert^{\frac{1}{\alpha+1}}e^{-(k+l)P(\sigma_Y, \Phi)}\\
& \geq C_1 \frac{(\vert \pi^{-1}[u_1]\vert \vert
\pi^{-1}[u_2]\vert) ^{\frac{1}{\alpha+1}}S_{k-n-2M}e^{(n-k)P(\sigma_Y,
\Phi)}}{e^{(n+l)P(\sigma_Y, \Phi)}} \textnormal{ (by Lemma \ref{topmix5})}
 \\&
\geq  C_{1}{C_{2}}^{-2} \nu[u_1]\nu[u_2]
\frac{S_{k-n-2M}}{e^{(k-n)P(\sigma_Y, \Phi)}} \geq
\frac{C_{1}\nu[u_1]\nu[u_2]}{{C_{2}}^{2}e^{2MP(\sigma_Y, \Phi)}}.
\end{split}
\end{equation*}
\end{proof}
\begin{proposition}\label{uniqued}
Let $(X,\sigma_X),(Y, \sigma_Y)$ be topologically mixing shifts of
finite type and $\pi:(X, \sigma_X)\rightarrow (Y,\sigma_Y)$ a
factor map. Let $\alpha>0.$ For all $n\in \N$, define
$\phi_n:Y\rightarrow \R$ by ${\phi_n}(y)= \log \vert \pi^{-1}[y_1
\cdots y_n] \vert^{1/(\alpha+1)}$ and
$\Phi=\{\phi_n\}^{\infty}_{n=1}.$  For each $\alpha>0$, there is a
unique equilibrium state for $\Phi.$ The unique measure is Gibbs
and mixing.
\end{proposition}
\begin{remark}\label{runiqued}
We have the same results for $\alpha=0.$ The unique equilibrium
state $\mu$ is not always Gibbs for a continuous function but it
is Gibbs for $\Phi$ (see Example \ref{notgibbsforc}).
\end{remark}
\begin{proof}
We follow the proof of Theorem 5 \cite{B2006} and so we only give
an outline of the proof. In Lemma \ref{equistateY}, we construct
an equilibrium state $\mu$ for $\Phi$ which is Gibbs. If $\mu$ and
$\mu'$ are two distinct $\sigma_Y$-invariant ergodic Borel
probability measures on $Y$, then $\mu$ and $\mu'$ must be
mutually singular. Therefore, by Lemma \ref{ergodic}, $\mu$ in
Lemma \ref{equistateY} is unique  and it is the unique ergodic
invariant Borel probability measure on $Y$ that satisfies the
Gibbs property for $\Phi.$ The same arguments as in the proof of
Theorem 5 \cite{B2006} show that  $\mu$ is the unique equilibrium
state for $\Phi$ and that it is mixing.
\end{proof}
In order to study uniqueness of the equilibrium state for $-\alpha
\Phi\circ \pi =\{-\alpha \phi_n\circ \pi\}^{\infty}_{n=1},$ we use
the following results by Petersen, Quas and Shin \cite{PQS}.
\begin{theorem} \cite{PQS} \label{coPQS}
Let $(X, \sigma_X)$ be an irreducible shift of finite, $(Y,
\sigma_Y)$ a subshift and $\pi:(X, \sigma_X) \rightarrow (Y,
\sigma_Y)$ a one-block factor map. Suppose that $\pi$ has a singleton clump
$a$, i.e., there is a symbol $\{a\}$ of $Y$ such that the number
of preimages of $\{a\}$ under $\pi$ is one in $Y.$ Then every
$\sigma_Y$-invariant ergodic Borel probability measure on $Y$ which assigns positive measure to $[a]$ has
a unique preimage of maximal entropy.
\end{theorem}
The following results generalize the results on uniqueness of the
invariant ergodic measure of full dimension in \cite{Yayama}.
\begin{theorem}\label{uniqueness}
Let $(X,\sigma_X), (Y, \sigma_Y)$ be topologically mixing shifts
of finite type and $\pi:(X,\sigma_X)\rightarrow (Y,\sigma_Y)$ a
factor map.  Let $K$ be an SFT-NC carpet corresponding to $(X,
Y,\pi).$ If $\pi$ has a singleton clump, then there is a unique
$T$-invariant ergodic measure of full dimension for $K$. Define $\Phi
=\{\phi_n\}^{\infty}_{n=1}$ as in Proposition \ref{uniqued}. If
$\Phi$ is an almost additive potential on $(Y,\sigma_Y),$ then
there is a unique $T$-invariant ergodic measure of full dimension for $K$ and it is both
Gibbs and mixing.
\end{theorem}
\begin{remark}
$\Phi$ is not always almost additive even when $\pi$ has a
singleton clump (Example \ref{eshin}). $\Phi$ can be almost
additive even when $\pi$ has no singleton clump (Example
\ref{nosc}).
\end{remark}
\begin{proof}  By Theorem \ref{HDDFORSFT}, we
identify the $T$-invariant ergodic measures of full dimension with
the ergodic equilibrium states for $-\alpha \Phi\circ\pi.$ If
$\Phi$ is almost additive, then $-\alpha \Phi\circ\pi$ is almost
additive. Therefore, in this case, applying Theorem \ref{BM},
there is a unique equilibrium state and the unique measure is
Gibbs and mixing. If $\pi$ has a singleton clump, first notice by
Proposition \ref{sftwentropy} that  the equilibrium states are the
measures of maximal relative entropy over the unique measure $\nu$
for $\Phi.$ Since by Proposition \ref{uniqued} the unique
equilibrium state $\nu$ is Gibbs, if $1$ is the singleton clump,
then $\nu([1])>0.$ Applying Theorem \ref{coPQS}, we conclude that
there is a unique equilibrium state for $-\alpha \Phi\circ\pi.$
Therefore, there is a unique $T$-invariant ergodic measure of full dimension.
\end{proof}
\section{Examples}
In this section, we give examples that illustrate the results on sections \ref{main1} and \ref{application}.
Let $(X,\sigma_X), (Y, \sigma_Y)$ be subshifts and $\pi:(X,\sigma_X)\rightarrow (Y,\sigma_Y)$ a
factor map.  Let $\alpha>0.$ For all $n\in \N$, define
$\phi_n:Y\rightarrow \R$ by ${\phi_n}(y)= \log \vert \pi^{-1}[y_1
\cdots y_n] \vert^{1/(\alpha+1)}, \phi^{0}_{n}:Y\rightarrow \R$ by
$\phi^{0}_{n}(y)=\log \vert \pi^{-1}[y_1y_2 \cdots y_n]\vert$ and
$\psi_{n}:Y\rightarrow \R$ by
$\psi_{n}(y)=\log \vert D_n(y)\vert.$ Let
$\Phi=\{\phi_n\}^{\infty}_{n=1}$ and
$\Phi_{0}=\{\phi^{0}_{n}\}^{\infty}_{n=1}$. Define
$F:Y\rightarrow \R$ by $F(y)=P(\sigma_X,\pi,0)(y)$ for all $y \in Y.$
\begin{example} \cite{S1} (Singleton clump case without $\Phi_0$ being almost additive)\label{eshin}
\end{example}
This example appeared in Example 3.1 \cite{S1} and it was shown
that there is no {\em continuous} saturated compensation function.
Applying Theorems \ref{iPS} and \ref{thm1}, we find a Borel
measurable saturated compensation function $-F\circ \pi,$ where
$F=\lim_{n\rightarrow \infty}(1/n)\phi^{0}_{n}$ a.e. with respect
to every $\sigma_Y$-invariant Borel probability measure on $Y$. We
claim  that $\Phi_{0}$ is not almost additive on $(Y,\sigma_Y)$.
Let $X\subset \{1, 2, 3, 4,5\}^{\N}$ and $Y\subset \{1, 2\}^{\N}$
be the shifts of finite type determined by the transitions given
by Figure \ref{fig13}. Define $\pi$ by $\pi(1)=1,$
$\pi(2)=\pi(3)=\pi(4)=\pi(5)=2.$
\bigskip
\begin{center}
\xymatrix{ & & 5  \ar@{->}[rr]  \ar@{<->}[ddrr] & &   1  \ar@{->}
[dd] \ar@{<->}[rr]  & &   2 \ar@(ur,dr)
 & & & &
1 \ar@{<->}[dd] \\
& & & & & &  \ar [r]^\pi & &  \\
& & 4  \ar@{<->}[rr] & & 3  & &   & & & & 2 \ar@(ur,dr)}
\end{center}
\begin{figure}[h]
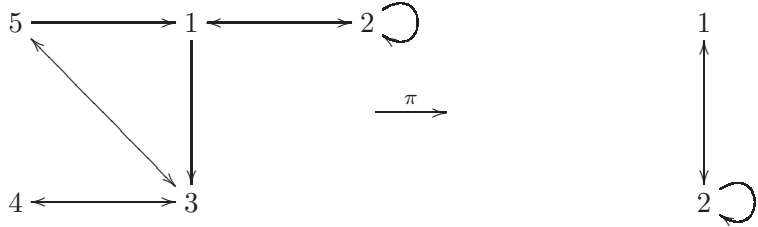

{} \caption{$X,Y,$ and $\pi$ in Example \ref{eshin}} \label{fig13}
\end{figure}
In \cite{S1}, it was shown that for $n$ odd $\vert
\pi^{-1}[12^n1]\vert=1$ and for $n=2k$ even $\vert
\pi^{-1}[12^{2k}1]\vert =2^{k-1}+1.$ For $n$ odd, $\vert
\pi^{-1}[12^n1]\vert /(\vert \pi^{-1}[12]\vert \vert
\pi^{-1}[2^{n-1}1]\vert)$ is clearly not bounded below by a
positive constant. Therefore $\Phi_{0}$ is not almost additive.
Applying Theorem \ref{uniqueness}, an NC-SFT carpet corresponding
to $(X, Y, \pi)$
has a unique $T$-invariant ergodic measure of full dimension. The Hausdorff dimension of the set is given by the formulas in Theorem \ref{HDDFORSFT}.
\begin{example}\label{ythesis}(Singleton clump case with $\Phi_0$ being almost additive)
\end{example}
Let $X\subset \{1, 2, 3, 4\}^{\N}$ and $Y= \{1, 2\}^{\N}$ be the
shifts of finite type determined by the transitions given by
Figure \ref{fig12}. Define $\pi$ by $\pi(1)=1,$
$\pi(2)=\pi(3)=\pi(4)=2.$ We cannot apply theorems in
\cite{Yayama} to find a saturated compensation function
$G\circ\pi, G\in (Y)$, because the transition matrix among symbols
of the preimages of $\{2\}$ is not primitive and so $G$ is not
defined. Applying Theorems \ref{iPS} and \ref{thm1}, we find a
Borel measurable saturated compensation function $-F\circ \pi,$
where $F=\lim_{n\rightarrow \infty}(1/n)\phi^{0}_{n}$ a.e. with
respect to every $\sigma_Y$-invariant Borel probability measure on
$Y$. We claim that $\Phi_{0}$ is almost additive on $(Y,
\sigma_Y)$.
\bigskip
\begin{center}
\xymatrix{ & & 1 \ar@(ul,dl) \ar@{<-}[ddrr] \ar@{<->}[rr]
\ar@{<-}[dd]& &   2 \ar@{<->} [ddll]   & &
 & & 1 \ar@{<->}[dd]  \ar@(ur,dr )\\
& & & & & &   \ar [r]^\pi & &  \\
& & 3   \ar@{<->}[rr]& & 4  & & & & 2 \ar@(ur,dr )}
\end{center}
\bigskip
\begin{figure}[h]
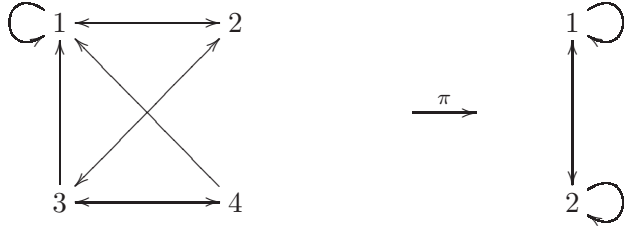

{} \caption{$X,Y$ and $\pi$ in Example \ref{ythesis}}
\label{fig12}
\end{figure}
\noindent Let $a$ be the largest eigenvalue of the irreducible
matrix among symbols of the preimages of $\{2\}.$ Notice that
$\vert \pi^{-1}[1^{l_1}2^{n_1}1^{l_2}2^{n_2}1^{l_3}]\vert=\vert
\pi^{-1}[12^{n_1}1]\vert \vert \pi^{-1}[12^{n_2}1]\vert.$ Define
$1^{0}2^{n}=2^{n}1^{0}=2^{n}$. It is enough to show that for
$n=k_1+k_2, k_1, k_2\geq 1$ there exists $C_1, C_2>0$ such that
$$C_1 \leq \frac{\vert\pi^{-1}[1^{i}2^{n}1^{j}]\vert}{{\vert\pi^{-1}[1^{i}2^{k_1}]\vert \vert\pi^{-1}[2^{k_2}1^{j}]\vert}}, \frac{\vert\pi^{-1}[12^{n}1^{j}]\vert}{{\vert\pi^{-1}[1]\vert \vert\pi^{-1}[2^{n}1^{j}]\vert}}, \frac{\vert\pi^{-1}[1^{i}2^{n}1]\vert}{{\vert\pi^{-1}[1^{i}2^n]\vert \vert\pi^{-1}[1]\vert}}\leq C_2,$$ for $0\leq i,j\leq1$.
Applying the Perron-Frobenius Theorem, we can find
$\widetilde{C_1},\widetilde{C_2}>0$ such that
$$\widetilde{C_1}a^{n-1}\leq \vert \pi^{-1}[12^{n}1]\vert, \vert
\pi^{-1}[2^{n}]\vert, \vert \pi^{-1}[12^{n}]\vert, \vert
\pi^{-1}[2^{n}1]\vert \leq \widetilde {C_2}a^{n-1} \textnormal{
for all } n\in\N.$$ Therefore,
 $\Phi_{0}$ is almost additive. By Theorem
\ref{uniqueness}, an NC-SFT carpet corresponding to $(X,Y,\pi)$
has a unique $T$-invariant ergodic measure of full dimension and
it is Gibbs and mixing. The Hausdorff dimension of the set is
given by the formulas in Theorem \ref{HDDFORSFT}.
\begin{example}(No singleton clump case)\label{nosc}
\end{example}
We construct $(X,Y,\pi)$ such that $\Phi$ is almost additive on
$(Y, \sigma_Y)$. Let $X\subset \{1, 2, 3, 4\}^{\N}$ and $Y= \{1,
2\}^{\N}$ be the shifts of finite type determined by the
transitions given by Figure \ref{fig11}.
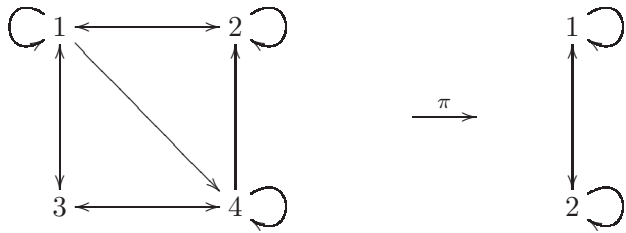
\bigskip
\begin{center}
\xymatrix{ & & 1 \ar@(ul,dl) \ar@{->}[ddrr] \ar@{<->}[rr]
\ar@{<->}[dd]& &   2 \ar@{<-} [dd]  \ar@(ur,dr) & &
 & & 1 \ar@{<->}[dd]  \ar@(ur,dr )\\
& & & & & &   \ar [r]^\pi & &  \\
& & 3   \ar@{<->}[rr]& & 4 \ar@(ur,dr) & & & & 2 \ar@(ur,dr )}
\end{center}
\bigskip
\begin{figure}[h]
{} \caption{$X,Y$ and $\pi$ in Example \ref{nosc}} \label{fig11}
\end{figure}
Define $\pi$ by $\pi(1)=\pi(2)=1,$ $\pi(3)=\pi(4)=2.$ Let $A_1$ be
the transition matrix of symbols of 1 and 2 in $X$ and $A_2$ be
the transition matrix of symbols of 3 and 4 in $X.$ Then the
largest positive eigenvalue of $A_1$ is 2 and the largest positive
eigenvalue of $A_2$ is $(1+\sqrt 5)/2.$ Let $a=(1+\sqrt 5)/2.$
Clearly, we have $\vert \pi^{-1}[12^n1]\vert =\vert
\pi^{-1}[2^n]\vert$ for all $n\in \N$ and  $\vert
\pi^{-1}[12]\vert =\vert \pi^{-1}[21]\vert=\vert\pi^{-1}[2]\vert$.
It is easy to see that for  $t\geq 1, k_1 \cdots k_t, l_1,\cdots
l_t \geq 1,$
$$\vert \pi ^{-1}[2^{k_1}1^{l_1}\cdots 2^{k_t}1^{l_t}]\vert=\vert
\pi ^{-1}[1^{l_1}2^{k_1}\cdots 1^{l_{t}}2^{k_t}]\vert
=2^{l_1+\cdots l_t-2(t-1)-1}\vert\pi^{-1}[2^{k_1}]\vert \cdots
\vert\pi^{-1}[2^{k_t}]\vert,$$ and for $t\geq 2,$
\begin{equation*}
\begin{split}
&\vert \pi ^{-1}[2^{k_1}1^{l_1}\cdots 1^{l_{t-1}}2^{k_t}]\vert=2^{l_1+\cdots l_{t-1}-2(t-1)}\vert\pi^{-1}[2^{k_1}]\vert \cdots \vert\pi^{-1}[2^{k_t}]\vert,\\
&\vert \pi ^{-1}[1^{l_1}2^{k_1}\cdots 2^{k_{t-1}}1^{l_t}]\vert=2^{l_1+\cdots l_{t}-2(t-2)-2}\vert\pi^{-1}[2^{k_1}]\vert \cdots \vert\pi^{-1}[2^{k_{t-1}}]\vert.\\
\end{split}
\end{equation*}
\noindent Applying the Perron-Frobenius Theorem, for $k \in \N,$
there exist $C_1, C_2>0$ such that $C_1 a^{k-1}\leq \vert
\pi^{-1}[2^k]\vert\leq C_2 a^{k-1}$.  Then calculations show
(\ref{aadditiveinourcase}) and so $\Phi$ is almost additive.
Applying Theorems \ref{iPS} and \ref{thm1},  $-F\circ \pi,$ where
$F=\lim_{n\rightarrow \infty}(1/n)\phi^{0}_{n}$ a.e. with respect
to every $\sigma_Y$-invariant Borel probability measure on $Y$, is
a Borel measurable saturated compensation function for $(\sigma_X,
\sigma_Y, \pi)$ and $\Phi_{0}$ is almost additive on $(Y,
\sigma_Y).$ By Theorem \ref{uniqueness}, an NC-SFT carpet
corresponding to $(X, Y,\pi)$ has a unique $T$-invariant ergodic
measure of full dimension. It is Gibbs and mixing. The Hausdorff
dimension of the set is given by the formulas in Theorem
\ref{HDDFORSFT}.
\begin{example}(A subshift with  $\psi_n$ being continuous for all $n\in \N$)\label{subshift1}
\end{example}
Let $X\subset \{1, 2, 3, 4\}^{\N}$ be the subshift with the
following forbidden blocks: 23,24,32,42,$12^{n}1$ for $n\geq 4$
and $1C_n1$ where $C_n$ is an allowable word of the full shift of
two symbols $\{3,4\}$ which has length $n\geq 4$. Define $\pi$ by
$\pi(1)=1,\pi(2)=\pi(3)=\pi(4)=2.$  Then $Y \subset \{1,2\}^{\N}$
is the subshift with the forbidden blocks $12^{n}1$ for $n\geq 4.$
It is easy to see that $\psi_n$ is continuous for all $n\in \N$.
Applying Theorem \ref{thm1}, $-F\circ\pi$ is a Borel measurable
saturated compensation function for $(\sigma_X,\sigma_Y,\pi).$ The
Hausdorff dimension of a subshift-NC carpet corresponding to
$(X,Y,\pi)$ is given by any formula in Theorem \ref{application1}
and the $T$-invariant ergodic measures of full dimension are the
ergodic equilibrium states for $-\alpha \Phi^{'}\circ \pi$, where
$\Phi^{'}$ is defined in Theorem \ref{application1}.
\begin{example}(A subshift with $\psi_n$ being discontinuous for all $n\in \N$)\label{subshift2}
\end{example}
Let $X\subset \{1, 2, 3, 4\}^{\N}$ be the subshift with the
following forbidden blocks: 23,24,32,42, $12^{2n+1}1$ for $n\geq
0$ and $1C_{2n}1$ where $C_{2n}$ is an allowable word of the full
shift of two symbols $\{3,4\}$ which has length $2n,n\geq1.$
Define $\pi$ by $\pi(1)=1,\pi(2)=\pi(3)=\pi(4)=2.$ Then $Y
=\{1,2\}^{\N}$. We show that $\psi_n$ is not continuous for all
$n\in \N$. For $n\geq 3,$ let $y=2^{n-2}12^{\infty}\in Y$. Then
$\vert D_n(y)\vert=3\vert\pi^{-1}[2^{n-2}1]\vert$. Let
$z_1=2^{n-2}12^{2k}1\cdots \in Y$ for $k\in \N$. Then $\vert
D_n(z_1)\vert=\vert\pi^{-1}[2^{n-2}1]\vert$ for all $k\in \N.$ Let
$z_2=2^{n-2}12^{2k+1}1\cdots \in Y$ for $k\in \N$. Then $\vert
D_n(z_2)\vert=2\vert\pi^{-1}[2^{n-2}1]\vert$ for all $k\in \N.$
Therefore $\psi_n$ is not continuous at $y=2^{n-2}12^{\infty}$ for
$n\geq 3$. Similarly, for $n=1,2,\psi_n$ is not continuous at
$y=12^{\infty}.$  Applying Theorem \ref{thm1}, $-F\circ\pi$ is a
Borel measurable saturated compensation function for
$(\sigma_X,\sigma_Y,\pi).$ The Hausdorff dimension of a
subshift-NC carpet corresponding to $(X,Y,\pi)$ is given by the
formulas (\ref{HDformula1.1}) and (\ref{HDformula1.2}) and the
$T$-invariant ergodic measures of full dimension are the ergodic
equilibrium states for $-\alpha \Phi^{'}\circ \pi$, where
$\Phi^{'}$ is defined in Theorem \ref{application1}.
\begin{example} \label{notgibbsforc}(On Remark \ref{runiqued})
\end{example}
In this example, we find $(X,Y, \pi)$ such that there is a unique
equilibrium state $\nu$ for $-G\in C(Y)$ which is not Gibbs, where
$G\circ\pi $ is a saturated compensation function. We will see
that $\nu$ is the unique equilibrium state for $\Phi_0$ which is
Gibbs.

Let $X\subset \{1, 2, 3\}^{\N}$ and $Y\subset \{1, 2\}^{\N}$ be
the shifts of finite type determined by the transitions given by
Figure \ref{fig14}. Define $\pi$ by $\pi(1)=1,$ $\pi(2)=\pi(3)=2.$
\bigskip
\begin{center}
\xymatrix{
& & & &     2 \ar@(ur,dr)[] \ar[dd]  \\
& & 1\ar@{<->}[urr] \ar@{<-}[drr] & & & & \ar [r]^\pi & &
1\ar@{<->}
[r] & 2 \ar@(ur,dr) \\
& & & &    3 \ar@(ur,dr)[]            }
\end{center}
\begin{figure}[h]
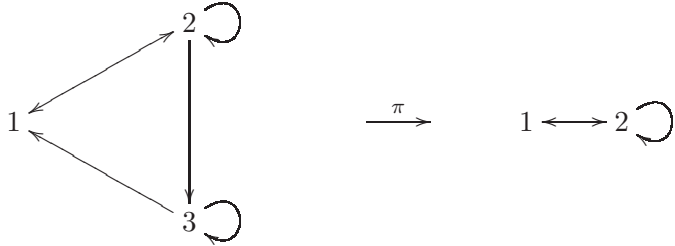

{} \caption{$X,Y,$ and $\pi$ in Example \ref{notgibbsforc}}
\label{fig14}
\end{figure}
By Theorem 3.1 \cite{Yayama}, there is a saturated compensation
function $G\circ\pi,$ where $G\in C(Y)$ is defined by
\begin{equation*}
G(y)= \begin{cases} \log ((n-1)/n) & \textrm{
 if }y\in[2^{n}1], n\geq 2 \\
0 & \textrm{
 if }y\in [1] \cup [21] \cup \{2^{\infty}\},
\end{cases}
\end{equation*}
and $-G$ has a unique equilibrium state $\nu$ which is not Gibbs.
Since by Theorem 3.4 \cite{Wcom},
$$-\int G d m=\int F d m \textnormal{ for all } m \in M(Y,\sigma_Y),$$ we obtain
$$\sup_{m \in M(Y,\sigma_Y)}\{h_{m}(\sigma_Y)-\int G dm \}= \sup_{m \in M(Y,\sigma_Y)}\{h_{m}(\sigma_Y)+\lim_{n\rightarrow\infty}\frac{1}{n}\int \phi^{0}_{n} dm\}.$$
A simple calculation shows that $\Phi_0$ is not almost additive on
$(Y, \sigma_Y).$ By the first part of Remark \ref{runiqued}, there is a unique
equilibrium state for $\Phi_0$ and it is Gibbs. It coincides with
the unique equilibrium state $\nu$ for $-G.$
\section{Problems}
 We finish by mentioning a couple of questions related to the results in section \ref{application}.
Firstly, when is $\phi_n^{'}$ in Theorem \ref{application1} continuous for all $n\in \N$ (see Examples \ref{subshift1} and \ref{subshift2})? Secondly, does the formula (\ref{subshiftwc}) in Theorem \ref{application1} hold without $\phi_n^{'}$ being continuous?\\

{\em Acknowlegements.}
I would like to thank Professor Alejandro Maass and Professor Karl Petersen for helpful discussions and advice.
This research is supported by Fondecyt Postdoctoral Grant N$^{0}3090015$ and Basal Grant at the Centro de Modelamiento Matem\'{a}tico, Universidad de Chile.

\end{document}